\documentclass[reqno,11pt,a4]{amsart}  
\usepackage{amsmath}
\usepackage{amssymb} 
\usepackage{mathtools}   
\usepackage{graphicx}
\usepackage{graphics}  
\usepackage{color}
\usepackage{verbatim} 
\usepackage[normalem]{ulem} 
\usepackage{upgreek}
\usepackage{enumerate} 
\usepackage{cite}
\usepackage{amsfonts}
\usepackage{bbm}
\usepackage{graphicx}
\usepackage{multicol}
\usepackage[utf8]{inputenc}
\usepackage{framed}
\usepackage{amsthm}
\usepackage{hyperref}
\usepackage{caption}
\usepackage{subcaption}
\usepackage[capitalize]{cleveref}
\usepackage[titletoc]{appendix}
\numberwithin{equation}{section}  

\usepackage[refpage,noprefix]{nomencl}\usepackage{nomencl} 

\newtheorem{theorem}{Theorem}[section] 
\newtheorem{lemma}[theorem]{Lemma} 
\newtheorem{proposition}[theorem] {Proposition} 
\newtheorem{cor}[theorem]{Corollary} 
\newtheorem{remark}[theorem]{Remark} 
\newtheorem{definition}[theorem] {Definition} 
 
\newtheorem{assump}[theorem]{Assumption}

\theoremstyle{definition}

\usepackage{float}
\restylefloat{figure}

\DeclareMathAlphabet{\mathpzc}{OT1}{pzc}{m}{it}

\newcommand{\abs}[1]{\left| #1 \right|}

\renewcommand{\L} {\Lambda} %

\def\d{\delta} 
\newcommand{\e} {\varepsilon} 
\newcommand{\eps}{\varepsilon} 
\renewcommand{\epsilon}{\varepsilon}

\def\l{\lambda}

\newcommand{\tY}{\tilde{Y}}

\newfam\Bbbfam 
\font\tenBbb=msbm10 
\font\sevenBbb=msbm7 
\font\fiveBbb=msbm5 
\textfont\Bbbfam=\tenBbb 
\scriptfont\Bbbfam=\sevenBbb 
\scriptscriptfont\Bbbfam=\fiveBbb

\newcommand{\R}     {\mathbb{R}} 
 
\newcommand{\N}     {\mathbb{N}} 
\renewcommand{\P}   {\mathbb{P}} 
 
\newcommand{\E}     {\mathbb{E}} 
\newcommand{\Q}     {\mathbb{Q}} 
 
 \newcommand{\floor}[1]{\left\lfloor {#1} \right\rfloor}

\def\1{{\mathchoice {1\mskip-4mu\mathrm l}      
{1\mskip-4mu\mathrm l} 
{1\mskip-4.5mu\mathrm l} {1\mskip-5mu\mathrm l}}} 
\newcommand{\ssup}[1] {{\scriptscriptstyle{({#1}})}} 
\def\comment#1{} 
\newtheoremstyle{thm}{2ex}{2ex}{\itshape\rmfamily}{} 
{\bfseries\rmfamily}{}{1.7ex}{} 
\newtheoremstyle{rem}{1.3ex}{1.3ex}{\rmfamily}{} 
{\itshape\rmfamily}{}{1.5ex}{} 
 
\newcommand{\bE} {\mathrm{\textbf{E}}}

\newcommand{\bP} {\mathrm{\textbf{P}}} 
 
\newcommand{\bT} {\mathrm{\textbf{T}}}

\newcommand{\bW} {\mathrm{\textbf{W}}}

\newcommand{\bZ} {\boldsymbol{Z}} 
\newcommand{\ba}{\boldsymbol{p}}

\newcommand{\zero} {\boldsymbol{0}}

\newcommand{\Ccal}   {{\mathcal C }}

\newcommand{\Fcal}   {{\mathcal F }} 
\newcommand{\Gcal}   {{\mathcal G }} 
\newcommand{\Hcal}   {{\mathcal H }}

\newcommand{\Ocal}   { {O}}

\newcommand{\ex}{{\rm e}} 
 
\renewcommand{\d}{{\rm d}}

\newcommand{\Exp}{\mathscr{E}\kern-0.2mm{\operatorname{xp}}}
\newcommand{\Log}{\mathscr{L}\kern-0.2mm{\operatorname{og}}}

\newcommand\NoBlackBoxes{\global\overfullrule0pt}
\NoBlackBoxes
\newcommand\mycom[2]{\genfrac{}{}{0pt}{}{#1}{#2}}

\newcommand{\ek}[1]{\left[#1\right]}
\newcommand{\rk}[1]{\left(#1\right)}

\newcommand{\gk}[1]{\left\{#1\right\}}

\newcommand{\hk}[1]{^{(#1)}}

\newcommand{\Pt}{\P_{\mathbf{T}}}
\newcommand{\Et}{\E_{\mathbf{T}}}

\usepackage{xifthen}
\renewcommand{\root}{\zero}
\newcommand{\optarg}[1][]{%
  \ifthenelse{\isempty{#1}}%
    {{\tt{P}}_\L}
    {{\tt{P}}_\L^{\ssup{{\rm #1}}}}
}

\newcommand{\CC}{C_{\alpha}}

\newcommand{\phit}[2]{\phi_\bT\hk{#1}(#2)}
\newcommand{\nk}{{n_k}}
\newcommand{\Ynk}{Y_{n_k}}
\newcommand{\lrd}{\xleftrightarrow{\,*\,}}

\newcommand{\OpenEd}{y}

\newcommand{\cadlag}{c\`adl\`ag }

\newcommand{\Ypron}[1]{Y_{#1}\hk{n}}

\newcommand{\Var}[1]{\mathbf{Var}\left({#1}\right)} 
 
\author{Eleanor Archer}
\address{Modal'X, UMR CNRS 9023, UPL, Universit\'e Paris Nanterre, 92000 Nanterre, France}
\email{earcher@parisnanterre.fr}
\author{Quirin Vogel}
\address{Department of Mathematics, CIT, Technische Universität München, Boltzmannstr. 3, D-85748, Garching bei München, Germany.}
\email{quirinvogel@outlook.com}
\usepackage{fancyhdr}
\begin{document}
\title[Quenched critical percolation on Galton--Watson trees]{Quenched critical percolation on Galton--Watson trees}
\maketitle



\begin{abstract}
    We consider critical percolation on a supercritical Galton--Watson tree. We show that, when the offspring distribution is in the domain of attraction of an $\alpha$-stable law for some $\alpha \in (1,2)$, or has finite variance, several annealed properties also hold in a quenched setting. In particular, the following properties hold for the critical root cluster on almost every realisation of the tree: (1) the rescaled survival probabilities converge; (2) the Yaglom limit or its stable analogue hold - in particular, conditioned on survival, the number of vertices at generation $n$ that are connected to the root cluster rescale to a certain (explicit) random variable; (3) conditioned on initial survival, the sequence of generation sizes in the root cluster rescales to a continuous--state branching process. This strengthens some earlier results of Michelen (2019) who proved (1) and (2) in the case where the initial tree has an offspring distribution with all moments finite.
    \smallbreak\noindent
    \emph{\keywordsname:} critical percolation; incipient infinite cluster; branching processes
\end{abstract}

\section{Introduction}\label{sec:intro}

Let $\bT$ be a supercritical Galton--Watson tree with law $\bP$ such that the law of its offspring distribution is in the domain of attraction of a stable law with parameter $\alpha\in(1,2)$, or has finite variance, and is supported on $\gk{1,2,\ldots}$. Suppose that $\mu>1$ is the mean of the offspring distribution. The aim of this paper is to obtain quenched results for critical (Bernoulli) percolation on $\bT$. 

Let $\bT_n$ be the vertices of generation $n$ in $\bT$. It is well-known that
\begin{equation}\label{eqn:W convergence}
    \bW_n := \frac{\abs{\bT_n}}{\mu^n}\to \bW\, ,
\end{equation}
as $n\to\infty$ in $\mathrm{L}^p$ for $1\le p<\alpha$ and almost surely \cite[Theorems 0 and 5]{bingham1974asymptotic}. Moreover, it was shown by Lyons \cite[Theorem 6.2 and Proposition 6.4]{lyons1990random} that $\bP$-almost surely, the critical percolation probability on $\bT$ is equal to $\frac{1}{\mu}$. Given $\bT$, we run critical percolation (i.e. with retention probability $\frac{1}{\mu}$) on the edges of $\bT$. Let $\Pt$ denote the (random) law of percolation on $\bT$ and consider the (critical) root cluster on $\bT$. Under the \textit{annealed} law (i.e. sampling first the tree then running the percolation process and then averaging over the tree), denoted $\P = \bP \circ \Pt$, the root cluster has the law of a critical Galton--Watson tree with offspring distribution in the domain of attraction an $\alpha$-stable law (with the same $\alpha$, or with $\alpha=2$ in the finite variance case). As a result, precise asymptotics for the annealed critical cluster are well-known. The purpose of this paper is to show that the same results hold in the \textit{quenched} setting, i.e. for $\bP$-almost every realisation of $\bT$. We note that some of our results were previously obtained by Michelen \cite{Michelen} under the assumption that the offspring distribution has all moments finite.

We start with a brief recap of known results in the annealed setting.

Firstly, for each $n \geq 0$ let $Y_n=\#\gk{v\in \bT_n\text{ connected to the root}}$, and set $\beta = \frac{1}{\alpha-1}$. Then there exists an explicit constant $\CC \in (0, \infty)$, depending on the offspring distribution, such that
\begin{equation}\label{eqn:slack prob convergence}
    n^{\beta}\P\rk{Y_n>0}\to \CC\, .
\end{equation}
This result was first obtained by Kolmogorov \cite{kolmogorov1938solution} in the case where the offspring distribution has has finite third moment, then by Kesten, Ney and Spitzer \cite{kesten1966galton} under a finite variance assumption, and was finally extended by Slack \cite[Lemma 2]{slack1968branching} to the stable window.

One can also say a lot more. In the finite variance case, Kesten, Ney and Spitzer \cite{kesten1966galton} also showed that the law of $n^{-1}Y_n$ conditioned on the event $Y_n>0$ converges in distribution to an exponential random variable. This is in fact known as \textit{Yaglom's limit} as it was first proved by Yaglom \cite{yaglom1947certain} under a third moment assumption. In the stable case, this was again extended by Slack \cite[Theorem 1]{slack1968branching} who showed that the rescaled random variables $\rk{n^{-\beta}Y_n}_{n\geq 1}$ under the conditioning $Y_n>0$ converge to an $\alpha$-stable random variable via the following convergence of the Laplace transforms:
\begin{equation}\label{eqn:slack Laplace convergence}
    \E{\ek{\ex^{-\theta n^{-\beta}Y_{n}} \middle\vert Y_{n}>0}} \to \phi (\theta) := 1-\CC^{-1}\theta (1+(\CC^{-1}\theta)^{\alpha-1})^{-\beta}.
\end{equation}
Moreover, in the finite variance case (in which case we set $\alpha=2$ and $\beta=1$), $Y$ has the exponential distribution with parameter $\CC$.

Note that, if $Y$ is a random variable with Laplace transform $\phi$ as above, then
\begin{equation}\label{eqn:Y mean}
    \E \ek{Y} = \lim_{\theta \to 0} \frac{1-\phi (\theta)}{\theta} = \CC^{-1}.
\end{equation}
One can in fact strengthen this last result quite dramatically to ask about convergence of the whole sequence of generation sizes under the conditioning that $Y_{n\epsilon}>0$ for some fixed $\epsilon>0$. Under this conditioning, the process $(n^{-\beta}Y_{n(t+\eps)})_{t \geq 0}$ converges to a \textit{continuous--state branching process} (CSBP) with initial condition determined by \eqref{eqn:slack Laplace convergence}. More precisely, the limiting process $(Y_t)_{t \geq 0}$ is a Markov process such that $Y_0$ is a random variable with Laplace transform given by $\phi (\theta \epsilon^{\beta})$, and for all $t,x, y \geq 0$ its transition kernels $P_{t}(x,y)$ satisfy the \textit{branching property}
\begin{equation}\label{eqn:branching prop}
    P_{t}(x+y,\cdot) = P_{t}(x,\cdot) \ast P_{t}(y,\cdot).
\end{equation}
The form of the transition kernel $P_t$ above is determined by a function known as the \textit{branching mechanism} which takes the form $\psi (\lambda) = \tilde{c}\lambda^{\alpha}$ in the stable case, and $\psi (\lambda) = \tilde{c}\lambda^{2}$ in the finite variance case, where $\tilde{c}$ is a constant depending on the precise form of the offspring distribution; we compute it in \cref{LemmaConstants}. We refer to \cite{li2012continuousstate} for further background on CSBPs.

Finally, rather than conditioning on the event $\{Y_n>0\}$, one can also look at the effect of conditioning on survival to infinity, and construct the corresponding law. In the annealed model, the law of the root cluster conditioned to survive forever has the law of a critical Galton--Watson process conditioned to survive forever and was constructed by Kesten \cite[Lemma 1.14]{kesten1986subdiffusive}. The critical cluster conditioned to survive is known as the \textit{incipient infinite cluster (IIC)}.

Our goal in this paper is to establish quenched versions of the above results. First steps in this direction were achieved by Michelen \cite[Theorem 1.3]{Michelen} who proved that a quenched version of \eqref{eqn:slack prob convergence} holds under a fourth moment assumption on the offspring distribution, and that a quenched version of \eqref{eqn:slack Laplace convergence} holds under the assumption that all moments are finite. In this paper we extend this to the general case of finite variance or stable tails. In particular we will work under the following assumption for all of our results.

\begin{assump}\label{assumption}
    Assume that the offspring distribution of $\bT$ supported on $\{1, 2, \ldots\}$, its mean is given by $\mu>1$, and that one of the following conditions holds.
    \begin{enumerate}[(a)]
        \item The offspring distribution of $\bT$ has finite variance. In this case set $\CC = \frac{2\E \ek{\abs{\bT_1}}^2}{\E \ek{\abs{\bT_1}\rk{\abs{\bT_1}-1}}}$ and $\beta=1$.
        \item The offspring distribution of $\bT$ has infinite variance with stable (power-law) tails, meaning that there exist $c_1 \in (0, \infty)$ and $\alpha \in (1,2)$ such that $\bP \rk{|\bT_1| \geq x} \sim c_1 x^{-\alpha}$ as $x \to \infty$. Here, we write $a_n\sim b_n$ for two sequences $\rk{a_n}_n$ and $\rk{b_n}_n$ if $a_n=b_n(1+o(1))$. In this case, we set $\beta=(\alpha-1)^{-1}$,
        \begin{equation}
            \CC=c^{-\beta}_1\mu^{\alpha\beta}\Gamma(1-\alpha)^{-\beta}\beta^\beta\qquad\text{and}\qquad \phi (\theta)=1-\CC^{-1}\theta (1+(\CC^{-1}\theta)^{\alpha-1})^{-\beta}\, .
        \end{equation}
    \end{enumerate}
\end{assump}
The derivation of the constant $\CC$ and the Laplace transform are in the Appendix; see Lemma \ref{LemmaConstants}. 
We would like to point out that we could also deal with the case where $\alpha=2$ and/or where there are slowly-varying corrections to the tail decay in case (b) above (using the exactly the same proofs), but have omitted this in order to lighten the notation. 

We furthermore note that the assumption that $\bT$ has no leaves is not a serious restriction since any supercritical Galton--Watson tree with leaves can be decomposed into a core consisting of a supercritical Galton--Watson tree with no leaves, to which several finite subcritical Galton--Watson trees are attached (this is the Harris decomposition; see \cite[Proposition 5.28]{lyons2017probability}). The p.g.f. of the law of the core as well as the distribution of the additional subcritical trees can be obtained explicitly from the original probability distribution. We anticipate that this can be used to extend our theorems to the case where leaves are permitted (though it requires some work and would impact the constants).

The first result is the following.
\begin{theorem}\label{thm:ConvergenceOfPartitionFunction}
For $\bP$-almost every $\bT$, we have that
\begin{equation}\label{eqn:partition prob convergence}
    n^{\beta}\Pt\rk{Y_n>0}\to \CC\bW\, .
\end{equation}
\end{theorem}

In addition, we obtain convergence of the rescaled generation sizes.

\begin{theorem}\label{thm:ConvergenceOfLaplaceTransform}
For $\bP$-almost every $\bT$, we have that the law of $Y_n$ conditioned on survival converges, i.e.
\begin{equation}
    \rk{Y_n|Y_n>0}\xrightarrow[n\to\infty]{(d)}Y\, ,
\end{equation}
where $Y$ has Laplace transform $\phi$ as in \eqref{eqn:slack Laplace convergence}.
\end{theorem}

We also add to this with convergence of the whole branching process. This was not considered in \cite{Michelen}.

\begin{theorem}\label{thm:ConvergenceOfBP}
For $\bP$-almost every $\bT$, under the conditioning $Y_n>0$ we have that the process $(n^{-\beta}Y_{n(t+1)})_{t \geq 0}$ converges in distribution to a continuous state branching process with the same branching mechanism $\psi$ appearing below \eqref{eqn:branching prop}, and with initial condition determined by \cref{thm:ConvergenceOfLaplaceTransform}. This convergence holds with respect to the Skorokhod-$J_1$ topology on the space $D([0, \infty), [0, \infty))$.
\end{theorem}

Finally, we turn to the IIC. Michelen shows in \cite[Lemma 3.9]{Michelen} that one can construct the IIC measure in the quenched setting essentially provided \cref{thm:ConvergenceOfPartitionFunction} holds (in his case this requires a fourth moment assumption; his proof works under our assumptions as well); in particular the law of the IIC, which we denote by $\mu_{\bT}$, satisfies 
\begin{equation}\label{eqn:IIC intro}
\mu_{\bT}(\Ccal_{\infty}[n]=t) = \frac{\sum_{v \in t_n}\bW_v}{\bW} \P_{\bT}(\Ccal[n]=t),
\end{equation}
where $\Ccal_{\infty}[n]$ denotes the IIC restricted to the first $n$ generations, $t_n$ denotes vertices in the $n^{th}$ generation of $t$, $\bW_v$ denotes the value of $\bW$ for the subtree rooted at $v$ and $\Ccal[n]$ denotes the first $n$ generations of an unconditioned critical cluster.

Michelen shows that under the assumption that all offspring moments are finite, the law of the size of the $n^{th}$ generation in $\Ccal_{\infty}$ is a size-biased version of the law appearing in \cref{thm:ConvergenceOfLaplaceTransform}. We establish a similar result under our assumption.

To state it, we first let $Y$ denote the random variable with Laplace transform appearing in \eqref{eqn:slack Laplace convergence}, and let $Y^*$ denote its size-biased version, meaning that (also using \eqref{eqn:Y mean}):
\begin{equation}
    \P \rk{Y^* \in [a,b]} = \frac{\E \ek{Y \mathbbm{1}\{Y \in [a,b]\}}}{\E \ek{Y}} = \CC\E \ek{Y \mathbbm{1}\{Y \in [a,b]\}}.
\end{equation}
We let $\bZ_n$ denote $n^{-\beta}Y_n$ under the law $\mu_{\bT}$.

\begin{theorem}\label{thm:IIC size biased}
For $\bP$-a.e. $\bT$, we have that (under $\mu_\bT$)
    \begin{equation}
       \bZ_n \xrightarrow[n\to\infty]{(d)}Y^*\, .
    \end{equation}
\end{theorem}

We anticipate that an analogue of \cref{thm:ConvergenceOfBP} should also hold under $\mu_{\bT}$, with convergence to a CSBP conditioned to survive, but we have not proved this here.

The main observation that allows us to strengthen the results of Michelen is that it in fact suffices to prove the almost sure convergence along an appropriate subsequence $(n_k)_{k \geq 0}$, and then extend to all $n$ using continuity properties of the various probabilities and processes. For \cref{thm:ConvergenceOfPartitionFunction} this continuity is an immediate consequence of the monotonicity of the connection probabilities, and the subsequential convergence simply follows from a refinement of the arguments of Michelen. The proof of \cref{thm:ConvergenceOfLaplaceTransform} is quite different to that of Michelen, however, who used a martingale originally studied in \cite{michelen2020quenched} to apply the method of moments. Instead, we use a byproduct of the proof of \cref{thm:ConvergenceOfPartitionFunction} which tells us that for each $n \geq 0$, we can choose $m$ satisfying $0 \ll m \ll n$ and such that, with high probability on the event $\{Y_n>0\}$, there will only a be a single vertex at level $m$ that connects directly upwards to level $n$. Therefore, by averaging over the choice of this vertex and the subtree emanating from it, we can establish good enough concentration of the Laplace transform to show that \cref{thm:ConvergenceOfLaplaceTransform} holds along a subsequence. The bulk of the proof is devoted to proving continuity estimates for the process $(Y_n)_{n \geq 0}$ which allow us to extend this to convergence for all $n$. The strategy used to prove \cref{thm:ConvergenceOfBP} is very similar. The proof of \cref{thm:IIC size biased} is also quite different to that of Michelen, who used Chebyshev's inequality to directly analyse the expression appearing in \eqref{eqn:IIC intro}. This is not possible for us since under our assumptions $\bW$ does not in general have a finite second moment; however by some careful analysis we are still able to evaluate the limiting expression in \eqref{eqn:IIC intro} to obtain almost sure subsequential convergence, and once again lift this to full convergence using similar ideas to the previous theorems.

This paper also leaves some remaining questions open. In particular we have not considered the \textit{genealogy} of a critical cluster. In the annealed model it is known that the Gromov--Hausdorff--Prokhorov scaling limit of critical cluster (i.e. when viewing the cluster as a metric-measure space) conditioned to be large is either the Brownian continuum random tree (in the case of finite variance offspring distribution) or a stable L\'evy tree (in the stable case). In light of our other theorems we anticipate that this should also be true in the quenched setting, but this would require somewhat different techniques to establish.

\textbf{Organisation of the paper.}
In Section \ref{sec:FirstTheorem} we prove Theorem \ref{thm:ConvergenceOfPartitionFunction} and introduce some notation used throughout the rest of the paper. Section \ref{sctn:Laplace transform proof}, we prove the convergence of the rescaled generation sizes given in Theorem \ref{thm:ConvergenceOfLaplaceTransform}. This includes a strategy to lift the convergence from a subsequence to the full sequence that will also be re-used in following sections. In Section \ref{sec:BrachningProcessProf}, we prove Theorem \ref{thm:ConvergenceOfBP}, the convergence of the full branching process. The measure of the IIC is then constructed in Section \ref{sec:IICProof} and Theorem \ref{thm:IIC size biased} is proved. Finally, in the Appendix we calculate the constant $\CC$ and give a large deviation estimate for the sum of independent random variables in the domain of attraction of a stable law.

\textbf{Acknowledgements.}
The authors would like to express their gratitude to the program \textit{Global Challenges for Women in Math Science} at the Technical University Munich which enabled a research visit of EA to Munich during which some of this work was written. EA would also like to thank Matan Shalev and Pengfei Tang for helpful conversations on this topic. EA was supported by the ANR project ProGraM, reference ANR-19-CE40-0025.

\section{Notation and set up}\label{sctn:notation}
Before we begin, we fix some notation that we will use throughout the paper. Let $(\Omega,\Fcal,\bP)$ be a probability space such that under $\bP$, $\bT$ has the law of a supercritical Galton--Watson tree with no leaves and mean number of descendants $\mu>1$. Let $\Fcal_n$ be the sigma-algebra generated by the Galton--Watson tree up to generation $n$. For a tree $\bT$, we define the probability space $\rk{\Omega^\bT,\Gcal^\bT,\Pt}$ in which $\Pt$ is the law of an independent Bernoulli percolation on the edges of $\bT$ with retention probability $\mu^{-1}$, $\Omega^\bT$ is the collection of all possible subtrees of $\bT$ and $\Gcal^\bT$ is the canonical sigma-algebra (generated by cylinder sets). We write $\Gcal_n^\bT$ for the sigma-algebra generated by the percolation process up to level $n$ in the tree. For events $A\in \Gcal^\bT$, $B\in\Gcal_n^\bT$ and $p\in (0,1)$, we abbreviate
\begin{equation}\label{eqn:sigma alg conditioning}
    \gk{\Pt\rk{A|\Gcal_n^\bT,B}<p}:=\gk{\sup_{\omega\in B}\Pt\rk{A|\Gcal_n^\bT}\!(\omega)< p}\, .
\end{equation}

For two subsets $A,B$ of the tree $\bT$, $A \leftrightarrow B$ means $A$ is connected to $B$ by open edges. $A \lrd B$ means $A$ is connected to $B$ by a path along which the distance to the root is monotone. We also let $\root$ denote the root of $\bT$. For $u,v \in \bT$, we let $|u|$ denote the distance of $u$ from $\root$, we let $u \wedge v$ denote the most recent common ancestor of $u$ and $v$, and we write $u \preceq v$ if $u$ is an ancestor of $v$.

We choose the following subsequence $n_k\in\N$ such that $n_k\sim k^{\frac{\sqrt{\alpha}+1}{\sqrt{\alpha}-1}}=k^A$, where $A={\frac{\sqrt{\alpha}+1}{\sqrt{\alpha}-1}}$. We remark that for some proofs, we could have chosen subsequences which grow slower (though still polynomially). However, we work with the subsquence $n_k$, in order to unify the different proofs.\\
Furthermore, unless stated otherwise, constants denoted by $c,C$ only depend on $\bP$ and can change from line to line.
\section{Convergence of the survival probabilities: proof of \cref{thm:ConvergenceOfPartitionFunction}}\label{sec:FirstTheorem}

In this section we prove \cref{thm:ConvergenceOfPartitionFunction}. 
\subsection{Outline of strategy and notation}
The proof of \cref{thm:ConvergenceOfPartitionFunction} is divided into a series of smaller lemmas.

\begin{lemma}[Convergence along a subsequence is sufficient]\label{lem:subsequence suff}
Fix $\e=\frac{\alpha-1}{2}$ and $p=\frac{\alpha-\e}{2}\in (0,\alpha/2)$. Set $A'=2p^{-1}\rk{\alpha-1}\e^{-1}=16/(\alpha+1)$. Note that $A'<A={\frac{\sqrt{\alpha}+1}{\sqrt{\alpha}-1}}$ and recall that $n_k\sim k^A$. Then convergence of \eqref{eqn:partition prob convergence} along the subsequence $(n_k)_{k \geq 1}$ implies convergence for all $n$. 
\end{lemma}
\begin{proof}
This follows from the fact that $n_{k+1}/n_k=1+o(1)$. Therefore if $n$ and $k$ are such that $n\in\ek{n_k,n_{k+1}}$, and the subsequential convergence holds, we have that
    \begin{equation}
        n^{\beta}\Pt\rk{Y_n>0}\le \rk{\frac{n_{k+1}}{n_k}}^\beta n_k^\beta\Pt\rk{Y_{n_k}>0}=\rk{1+o(1)}\CC\bW\, ,
    \end{equation}
and similarly for the lower bound.    
\end{proof}

In light of \cref{lem:subsequence suff}, we will prove almost sure convergence only along the subsequence $(n_k)_{k \geq 1}$. The proof will depend on the following inequality, which is a consequence of the inclusion-exclusion principle. For every $k \geq 1$ and $1 \leq m \leq \nk$, we have that
    \begin{equation}\label{eqn:inc exc}
        \abs{n_k^\beta\Pt\rk{Y_{n_k}>0}-n_k^\beta\sum_{v\in\bT_m}\Pt\rk{\root \leftrightarrow v\lrd \bT_{n_k}}}\le n_k^\beta\sum_{\mycom{u,v\in \bT_m}{u\neq v}}\Pt\rk{\root \leftrightarrow (u,v)\lrd \bT_{n_k}}\, ,
    \end{equation}
    where by $\root \leftrightarrow (u,v)\lrd \bT_{n_k}$ we mean that both $u$ and $v$ are connected to the root as well as to $\bT_{n_k}$.
    
Note that the sum on the left hand side counts the events in which there is a single vertex (connected to the root) at level $m$ that connects directly to level $n_k$, and the sum on the right-hand side counts the events in which there are at least two vertices (both connected to the root) at level $m$ that connect directly to level $n_k$.

    In what follows we will make an appropriate choice of $m$ so that the right-hand side of \eqref{eqn:inc exc} will be small, and the sum on the left hand side will be well-approximated by its mean, which is close to the desired limit.
    
In particular, for each $k \in \N$ we set
\begin{equation}\label{eqn:mk def}
    m_k = \lfloor  \frac{(1+\e)}{\rk{\alpha-1}\log\mu}\log n_k \rfloor.
\end{equation}
The proof of \cref{thm:ConvergenceOfPartitionFunction} then rests on the following two lemmas. In the first lemma we prove that, with this choice of $m_k$, the right-hand side of \eqref{eqn:inc exc} goes to $0$ almost surely. In the second lemma we prove that the sum appearing on the left hand side is well-approximated by its mean.

\begin{lemma}[Connection through a single node at lower level]\label{lem:connection single}
$\bP$-almost surely,
\begin{equation}
    \lim_{k\to\infty}n_k^\beta\sum_{\mycom{u,v\in \bT_{\tiny{m_k}}}{u\neq v}}\Pt\rk{\root \leftrightarrow (u,v)\lrd \bT_{n_k}}=0\, .
\end{equation}
\end{lemma}

\begin{lemma}[A concentration estimate]  \label{lem:variance}
$\bP$-almost surely,
\begin{equation}
   \lim_{k\to\infty} n_k^\beta\sum_{v\in\bT_{\tiny{m_k}}}\left[\Pt\rk{\root \leftrightarrow v\lrd \bT_{n_k}} - \mu^{-m_k} \P\rk{\root \leftrightarrow \bT_{n_k-m_k}} \right]= 0. 
\end{equation}
\end{lemma}

\cref{thm:ConvergenceOfPartitionFunction} for a subsequence follows immediately from the above three lemmas, \eqref{eqn:W convergence} and \eqref{eqn:slack prob convergence}, since they imply that
\begin{align*}
   \lim_{k \to \infty} n_k^\beta\Pt\rk{Y_{n_k}>0} = \lim_{k \to \infty} n_k^\beta\sum_{v\in\bT_{\tiny{m_k}}}\!\mu^{-m_k}\P\rk{\zero \leftrightarrow \bT_{n_k-m_k}} &= \lim_{k \to \infty}\frac{\abs{\bT_{\tiny{m_k}}}}{\mu^{m_k}}n^\beta_k\P\rk{\zero \leftrightarrow \bT_{n_k-m_k}} = \CC\bW  .
\end{align*}

\begin{remark}
    \label{rmk:no two vertices connect}
    Define $\ell_n$ so that $\ell_{n_k} = m_k$ for all $k \geq 1$ and extend to all $n$ by setting $\ell_n = \ell_{\nk}$ when $n \in [\nk, n_{k+1})$. 
Let $A_n$ be the event that there exist at least two vertices at level $m_n$ that connect the root to level $n$, and let $B_n$ be the event that there exist at least two vertices at level $\ell_n$ that connect the root to level $n$. By \cref{thm:ConvergenceOfPartitionFunction} and \cref{lem:connection single}, $\Pt\left(A_{n_k} \text{ i.o.}\middle\vert {Y_{n_k} >0}\right)\to 0$, $\bP$-almost surely. This also implies that $\Pt\left({B_{n} \text{ i.o.}}\middle\vert{Y_n >0}\right)\to 0$, $\bP$-almost surely. This follows since if $n \in [n_k, n_{k+1})$ and there are distinct $u, v$ in generation $\ell_n$ connecting the root to level $n$ using a path of length $n$, then since $\ell_n = m_k \leq n_k \leq n$ we obtain two distinct vertices in level $m_{k}$ that connect to level $n_k$ by following these two paths up to level $m_k$ and onwards to level $n_k$. Therefore, if $B_n$ occurred infinitely often, then so would $A_{n_k}$. Moreover, $\frac{\Pt \left(Y_{n_k} >0\right)}{\Pt\left(Y_{n} >0\right)} \to 1$, so we can exchange the two conditioning events.

This also explains why the limiting properties of the conditioned critical percolation cluster do not depend on the tree. Indeed, the behaviour of the cluster at level $n$ is essentially independent of the behaviour of the cluster before level $\ell_n$, since there will only be one subtree from level $\ell_n$ that will survive, so as $n$ goes to infinity we successively lose the dependence on more and more of the earlier generations in the tree.
\end{remark}

\subsection{Proof of lemmas}

In this subsection we prove \cref{lem:connection single} and \cref{lem:variance}.

\begin{proof}[Proof of \cref{lem:connection single}]
    (Connection through a single node at lower level.)
    
Recall that $A>2p^{-1}\rk{\alpha-1}\e^{-1}$, $\e=\frac{\alpha-1}{2}$ and $p=\frac{\alpha-\e}{2}=\frac{\alpha+1}{4}$, and that ${n_k} \sim k^A$. Note that
    \begin{equation}\label{eqn:p eps relation}
        \frac{p}{1-p}=\frac{1+\e}{1-\e}\, .
    \end{equation}
    Recall also that
    \[
m_k = \lfloor  \frac{(1+\e)}{\rk{\alpha-1}\log\mu}\log n_k \rfloor.
\]

We want to show that, $\bP$-almost surely,
\begin{equation}
    \lim_{k\to\infty}n_k^\beta\sum_{\mycom{u,v\in \bT_{\tiny{m_k}}}{u\neq v}}\Pt\rk{\root \leftrightarrow (u,v)\lrd \bT_{n_k}}=0\, .
\end{equation}
We abbreviate $m=m_k$ and $n=n_k$ for the rest of the proof. 

Recall that $\Fcal_i=\sigma\rk{\bT_r\colon 0\le r\le i}$ is the sigma algebra generated by the first $i$ levels of the tree. For $u,v\in \bT_m$, we set $\ba_{u,v} =\Pt\rk{u\lrd \bT_{n}}\Pt\rk{v\lrd \bT_{n}}$. By Markov's inequality, we have (taking $p$ as defined just above \eqref{eqn:p eps relation}) that
\begin{equation}\label{eqn:connection prob Markov}
    \begin{split}
        \bP\rk{\sum_{\mycom{u,v\in \bT_m}{u\neq v}}\Pt\rk{\root \leftrightarrow (u,v)\lrd \bT_{n}}>x}&\le \bE\ek{\rk{\sum_{\mycom{u,v\in \bT_m}{u\neq v}}\Pt\rk{\root \leftrightarrow (u,v)\lrd \bT_{n}}}^p}x^{-p}\\
        &\le \bE\ek{\bE\ek{\rk{\sum_{\mycom{u,v\in \bT_m}{u\neq v}}\Pt\rk{\root \leftrightarrow (u,v)\lrd \bT_{n}}}^p\Bigg|\Fcal_m}}x^{-p}\\
        &\le \bE\ek{\bE\ek{\rk{\sum_{\mycom{u,v\in \bT_m}{u\neq v}}\frac{\ba_{u,v} }{\mu^{2m-\abs{u\wedge v}}}}^p\Bigg|\Fcal_m}}x^{-p}\, ,
    \end{split}
\end{equation}
where $u \wedge v$ is the most recent common ancestor (i.e. furthest from the root) of $u$ and $v$, and $\abs{u\wedge v}$ is the distance of $u \wedge v$ from the root of $\bT$.

Given $u\in\bT_m$, let $\root=u_0,\ldots,u_m=u$ denote the ancestors of $u$ ordered by distance from the root. Let $\preceq $ denote the partial ordering induced by generations and let
\begin{equation}
    \bT_m\hk{u_i}=\gk{v\in \bT_m\colon u_i \preceq v}\, , 
\end{equation}
or in other words the set of vertices in $\bT_m$ that descend from $u_i$. Applying Markov's inequality, the tower property, summing over the most recent common ancestors and then reordering the sum, we deduce that
\begin{equation}
    \begin{split}
        {\bE\ek{\rk{\sum_{\mycom{u,v\in \bT_m}{u\neq v}}\frac{\ba_{u,v} }{\mu^{2m-\abs{u\wedge v}}}}^p\Bigg|\Fcal_m}}
        &\le {\bE\ek{\rk{\sum_{i=1}^m\sum_{u\in\bT_i} \sum_{\substack{a,b \in \bT_m\hk{u} \\ a \neq b}}\frac{\ba_{a,b} }{\mu^{2m-i}}}^p\Bigg|\Fcal_m}}\\
        &\le {\sum_{i=1}^m\sum_{u\in\bT_i}\bE\ek{\rk{\sum_{\substack{a,b \in \bT_m\hk{u} \\ a \neq b}}\frac{\ba_{a,b} }{\mu^{2m-i}}}\Bigg|\Fcal_m}^p}.
\end{split}
\end{equation}
(Note that $p<1$ for the final line.) Since each term of the form $\ba_{a,b} $ is independent of $\Fcal_m$ and moreover its conditional law does not depend on the choice of $a$ and $b$, this term can be factorised outside the internal expectation. 

Write $\overline{\bW}=\sup_n\bW_n$. Applying the tower property once more we deduce that there exists $C<\infty$ such that the expectation of this expression is upper bounded by (now letting $a,b$ denote arbitrary distinct vertices in $\bT_m$)
\begin{equation}
    \begin{split}
\bE\ek{\sum_{i=1}^m\sum_{u\in\bT_i}\rk{\frac{\rk{\bT_m\hk{u}}^2}{\mu^{2m-i}}}^p}\bE\ek{\ba_{a,b} }^p     
        &=\bE\ek{\sum_{i=1}^m\bE\ek{\sum_{u\in\bT_i}\mu^{-ip}\rk{\frac{\rk{\bT_m\hk{u}}^2}{\mu^{2m-2i}}}^p\Bigg|\Fcal_i}}\bE\ek{\ba_{a,b} }^p\\
        &\le \bE\ek{\sum_{i=1}^m\abs{\bT_i}\mu^{-ip}}\bE\ek{\overline{\bW}^{2p}}\bE\ek{\ba_{a,b} }^p\\
        &\le C \mu^{m(1-p)}\bE\ek{\overline{\bW}^{2p}}\bE\ek{\ba_{a,b} }^p\, .
    \end{split}
\end{equation}
Since $2p<\alpha$, we have that $\overline{\bW}^{2p}$ is bounded in $\mathrm{L}^1\rk{\bP}$ and hence, applying \eqref{eqn:slack prob convergence} to bound $\bE\ek{\ba_{a,b} }$ and substituting back into \eqref{eqn:connection prob Markov}, we obtain that there exists $C< \infty$ such that
\begin{equation}\label{eqn:BC prob single connection}
    \bP\rk{\sum_{\mycom{u,v\in \bT_m}{u\neq v}}\Pt\rk{\root \leftrightarrow (u,v)\lrd \bT_{n}}>x}\le C \mu^{m(1-p)}\rk{n-m}^{-2p/(\alpha-1)}x^{-p}\, .
\end{equation}
To prove the proposition, we choose $x = n^{-\beta}(\log n)^{-1}$. Recalling that $\mu^{m(1-p)} = n^{\frac{p(1-\epsilon)}{\alpha-1}}$, it follows that the right-hand side of \eqref{eqn:BC prob single connection} is upper bounded by $n^{-\frac{\eps p}{\alpha-1}}(\log n)^p$. By our choice of $A$, this bound is summable along the sequence $(n_k)_{k \geq 1}$ and hence the claim follows by Borel--Cantelli.
\end{proof}

We now turn to the proof of \cref{lem:variance}, which consists of a simple second moment estimate.

    \begin{proof}[Proof of \cref{lem:variance}] (A concentration estimate.)
Again write $n=n_k$ and $m=m_k$. Note that given $\Fcal_m$, we have that the collection of random variables
\begin{equation}
    \rk{\Pt\rk{v\lrd \bT_n}-\P\rk{\zero \leftrightarrow \bT_{n-m}}}_{v\in \bT_m}\, ,
\end{equation}
are i.i.d. with mean zero under $\bP$. It follows that
\begin{align}\label{eqn:first variance bound}
\begin{split}
    &\Var{\sum_{v\in\bT_m}\ek{\Pt\rk{0\leftrightarrow v\lrd \bT_n}-\mu^{-m}\P\rk{\zero \leftrightarrow \bT_{n-m}}}\Big|\Fcal_m}\\
&\qquad =\mu^{-2m}\sum_{v\in\bT_m}\Var{\Pt\rk{v\lrd \bT_n}} \leq \mu^{-2m}\sum_{v\in\bT_m}\bE \ek{\Pt\rk{v\lrd \bT_n}}\, .
\end{split}
\end{align}
Here the final inequality follows since $\Pt\rk{v\lrd \bT_n}^2 \leq \Pt\rk{v\lrd \bT_n}$. Note that, by \eqref{eqn:slack prob convergence}, the final line in \eqref{eqn:first variance bound} is bounded by
\begin{equation}\label{eqn:second variance bound}
    \Ocal\rk{\abs{\bT_m}(n-m)^{-\beta}\mu^{-2m}}\, .
\end{equation}
It therefore follows from Chebyshev's inequality that there exists $C<\infty$ such that for any $x>0$,
\begin{equation}\label{eqn:Cheb BC 1}
    \bP\rk{\left|\sum_{v\in\bT_m}\ek{\Pt\rk{0\leftrightarrow v\leftrightarrow \bT_n}-\mu^{-m}\P\rk{\zero \leftrightarrow \bT_{n-m}}} \right| >x\Big|\Fcal_m}\le C\frac{\abs{\bT_m}n^{-\beta}}{x^2\mu^{2m}}\, .
\end{equation}
Setting $x=n^{-\beta}(\log n)^{-1}$, it follows from the tower property and the fact that $\bE{[\overline{\bW}]}<\infty$ that there exists $C'<\infty$ such that
\begin{equation}\label{eqn:Cheb BC 2}
    \begin{split}
\bP&\rk{n^\beta\abs{\sum_{v\in\bT_m}\ek{\Pt\rk{0\leftrightarrow v\leftrightarrow \bT_n}-\mu^{-m}\P\rk{\zero \leftrightarrow \bT_{n-m}}}}>(\log n)^{-1}}\\
        &\le\bE\ek{C\frac{\abs{\bT_m}n^{-\beta}}{x^2\mu^{2m}}} \leq \frac{C'n^{-\beta}}{x^2\mu^{m}}=C'n^{-\frac{\e}{\alpha-1}}(\log n)^2\, ,
    \end{split}
\end{equation}
which is summable along the subsequence $\rk{n_k}_{k \geq 1}$ by our choice of $A$. Hence
\begin{equation}
      \bP\rk{n_k^\beta\abs{\sum_{v\in\bT_{\tiny{m_k}}}\ek{\Pt\rk{0\leftrightarrow v\leftrightarrow \bT_{n_k}}-\mu^{-m_k}\P\rk{\zero \leftrightarrow \bT_{n_k-m_k}}}}>(\log n)^{-1}\text{ infinitely often }}=0\, ,
\end{equation}
thus proving the lemma.
    \end{proof}

\section{Yaglom's limit: proof of \cref{thm:ConvergenceOfLaplaceTransform}}\label{sctn:Laplace transform proof}

In this section we examine the limit of the distribution of $n^{-\beta} Y_n$ conditioned on the event $\{Y_n>0\}$. We do this through a series of intermediate results, all concerning the Laplace transform. An advantage of our method is that we do not need to make any moment assumption on our offspring distribution. This contrasts to the proof of Michelen who proved a similar result using the method of moments, and therefore needed to assume that all moments were finite. We first state a preparatory lemma.

For $\theta\ge 0, n \geq 1$ let $\phi_n(\theta)$ be the critical (annealed) Laplace transform
\begin{equation}
    \phi_n(\theta)=\E\ek{\ex^{-\theta n^{-\beta}Y_n}|Y_n>0}\, .
\end{equation}
Recall from \eqref{eqn:slack Laplace convergence} that $\phi_n(\theta)$ converges to $\phi(\theta)$ as $n \to \infty$ for all $\theta \geq 0$, where
\[
\phi(\theta) = 1-\CC^{-1}\theta (1+\rk{\CC^{-1}\theta}^{\alpha-1})^{-\beta}.
\]

As in the proof of \cref{thm:ConvergenceOfPartitionFunction}, our strategy to prove \cref{thm:ConvergenceOfLaplaceTransform} is to start by proving almost sure convergence along a subsequence $(n_k)_{k \geq 1}$, and then extend this to full convergence by showing that $Y_n$ is very likely to be close to $Y_{n_k}$ if $n \in [n_k, n_{k+1}]$. We start with the subsequential convergence.

\begin{proposition}\label{PropositionConvergenceFixedTheta}
    Recall the definition of the subsequences $(n_k)_{k \geq 1}$ and $(m_k)_{k \geq 1}$ from \cref{lem:subsequence suff} and \eqref{eqn:mk def}. For every $\theta\ge 0$, we have that, $\bP$-almost surely
    \begin{equation}
\phit{n_k}{\theta} := \Et\ek{\ex^{-\theta n_k^{-\beta}Y_{n_k}}|Y_{n_k}>0} \to \phi(\theta)\, 
    \end{equation}
as $k \to \infty$.
\end{proposition}
    \begin{proof}
    By the monotonicity of the Laplace transform in $\theta\ge 0$, we get from \cref{thm:ConvergenceOfLaplaceTransform} and \eqref{eqn:BC prob single connection} that
    \begin{equation}
        \bP\rk{\exists \theta\ge 0\colon \sum_{u\neq v\in \bT_{\tiny{m_k}}}\Et\ek{\ex^{-\theta n_k^{-\beta}Y_{n_k}} \mathbbm{1}\{0\leftrightarrow (u,v)\lrd \bT_{n_k}\}|Y_{n_k}>0}>\frac{1}{\log(n_k)} \textnormal{ i.o. }}=0\, .
    \end{equation}
    Indeed, $\theta=0$ is the extremal case here, in which case we are reduced to \eqref{eqn:BC prob single connection}. By the inclusion-exclusion principle, this means that $\bP$-almost surely, we can write (henceforth writing $m$ and $n$ in place of $m_k$ and $n_k$ respectively)
    \begin{equation}\label{eqn:phi as average}
        \phit{n}{\theta}=\sum_{v\in \bT_m}\Et\ek{\ex^{-\theta n^{-\beta}Y^{(v)}_{n}}\mathbbm{1}\{0\leftrightarrow v \lrd \bT_{n}\} \middle| Y_{n}>0}+o(1)\, ,
    \end{equation}
    where $Y^{(v)}_{n}$ denotes elements in the root cluster at level $n$ that are also descendants of $v$.
    
    We now examine the unconditional Laplace transform:
    \begin{equation}
        \sum_{v\in \bT_m}\Et\ek{\ex^{-\theta n^{-\beta}Y^{(v)}_{n}}\mathbbm{1}\{0\leftrightarrow v \lrd \bT_{n}\}}\, .
    \end{equation}
    Note that given $\Fcal_m$, we have for any $v \in \bT_m$ that the collection of random variables
    \begin{equation}
        \rk{\Et\ek{\ex^{-\theta n^{-\beta}Y^{(v)}_{n}} \mathbbm{1} \{v\lrd \bT_{n}\}}-\E\ek{\ex^{-\theta n^{-\beta}Y_{n-m}} \mathbbm{1}\{Y_{n-m}>0\}}}_{v\in \bT_m}\, 
    \end{equation}
    are i.i.d. with mean zero under $\bP$. Hence, as in \eqref{eqn:first variance bound} and \eqref{eqn:second variance bound}, we have that
    \begin{multline}
        \Var{\sum_{v\in \bT_m}\ek{\Et\ek{\ex^{-\theta n^{-\beta}Y^{(v)}_{n}}\mathbbm{1}\{0\leftrightarrow v\lrd \bT_{n}\}}-\mu^{-m}\E\ek{\ex^{-\theta n^{-\beta}Y_{n-m}}\mathbbm{1}\{Y_{n-m}>0\}}}\Big|\Fcal_m}\\
        =\Ocal\rk{\abs{\bT_m}n^{-\beta}\mu^{-2m}}\, .
    \end{multline}
    Thus we have by Chebyshev's inequality and the first Borel--Cantelli lemma (exactly as in \eqref{eqn:Cheb BC 1} and \eqref{eqn:Cheb BC 2}) that eventually $\bP$-almost surely,
    \begin{multline}\label{eqn:Chebyshev io Laplace}
        \!n_k^\beta\abs{\sum_{v\in \bT_{\tiny{m_k}}}\!\ek{\Et\!\ek{\ex^{-\theta n_k^{-\beta}Y^{(v)}_{n_k}} \mathbbm{1}\{0\!\leftrightarrow\! v\lrd \bT_{n_k}\}}\!-\!\mu^{-m}\E\!\ek{\ex^{-\theta n_k^{-\beta}Y_{n_k-{m_k}}}\mathbbm{1}\{Y_{n_k-{m_k}}>0\}}}}\!<\!(\log n_k)^{-1}.
    \end{multline}
    To conclude, note that it follows by definition, continuity of $\phi$ and \eqref{eqn:slack prob convergence} that
    \begin{align}
    \begin{split}
        n_k^\beta\sum_{v\in \bT_m}\mu^{-m}\E\ek{\ex^{-\theta n_k^{-\beta}Y_{n_k-{m_k}}}\mathbbm{1}\{Y_{n_k-{m_k}}>0\}}&=n_k^\beta \bW_m\E\ek{\ex^{-\theta n_k^{-\beta}Y_{n_k-{m_k}}}\mathbbm{1}\{Y_{n_k-{m_k}}>0\}}\, \\
        &=\bW_m \CC \phi(\theta)\rk{1+o(1)}\, ,
    \end{split}
    \end{align}
    and hence combining with \eqref{eqn:phi as average}, \eqref{eqn:Chebyshev io Laplace} and \cref{thm:ConvergenceOfPartitionFunction} gives
\begin{align*}
   \phit{n_k}{\theta} &= \frac{\bW_{m_k} \CC}{n_k^\beta \Pt\rk{Y_{n_k}>0}} \phi(\theta)\rk{1+o(1)} \to \phi (\theta).\qedhere
\end{align*}
\end{proof}

Next, we verify that \cref{PropositionConvergenceFixedTheta} implies the almost sure convergence in law of $n_k^{-\beta}Y_{n_k}$ conditioned on survival.
\begin{cor}\label{CorollaryConvergenceCondtionedSubsequence}
    For $\bP$-almost every tree $\bT$, we have that $n_k^{-\beta}Y_{n_k}$ conditioned to survive up to time $n_k$ converges in distribution as $k \to \infty$ to a law with Laplace transform $\phi$ defined in \eqref{eqn:slack Laplace convergence}.
\end{cor}
\begin{proof}
    By Proposition \ref{PropositionConvergenceFixedTheta}, we have that
    \begin{equation}
        \bP\rk{\forall \theta\in [0,\infty)\cap\Q\colon \, \lim_{k\to\infty}  \phit{n_k}{\theta}=\phi(\theta)}=1\, .
    \end{equation}
    It just remains to remove the restriction $\theta\in\Q$. To this end, note that for every $\bT$, the function $\theta\mapsto\phit{n_k}{\theta}$ is monotone and that furthermore the limiting function $\phi$ is continuous. Hence, this immediately implies that
    \begin{equation}
        \bP\rk{\forall \theta\in [0,\infty)\colon \, \lim_{k\to\infty}  \phit{n_k}{\theta}=\phi(\theta)}=1\, ,
    \end{equation}
    and the result follows.
\end{proof}
It remains now to prove that the subsequential convergence of \cref{CorollaryConvergenceCondtionedSubsequence} can be lifted to the entire sequence $\rk{n^{-\beta}Y_n}_{n \geq 1}$. Our strategy is to work via a proof by contradiction: if $\abs{\phit{n}{\theta}-\phi(\theta)}$ was large but $\abs{\phit{n}{\theta}-\phit{n_k}{\theta}}$ was small, then $\abs{\phit{n_k}{\theta}-\phi(\theta)}$ would also be large, which would contradict Corollary \ref{CorollaryConvergenceCondtionedSubsequence}. Hence if we can show that $\abs{\phit{n}{\theta}-\phit{n_k}{\theta}}$ is consistently small, the result will follow.

We start in \cref{LemmaLink} by defining an event that indeed implies that $\abs{\phit{n}{\theta}-\phit{n_k}{\theta}}$ is small. Afterwards, we will explain how lower bounding the probability of this event directly implies \cref{thm:ConvergenceOfLaplaceTransform} and in particular in \cref{lem:lift to whole sequence} we will give a formal criterion of this ilk that does indeed imply the theorem. Then we just need to verify that the criterion is satisfied: this is the subject of \cref{prop:verify condition}.

\begin{lemma}[Upper bounding $|\phit{n}{\theta}-\phit{n_k}{\theta}|$]\label{LemmaLink}
    Fix some $\delta \in (0, \frac{\beta}{2})$. $\bP$-almost surely, we have for all $\theta>0$ and $\eps>0$ that there exists $K=K_{\delta, \epsilon, \theta, \bT} < \infty$ such that for all $k\ge K$ and all $n_{k-1}< n\le n_k$,
    \begin{equation}\label{eqn:event containment}
        \gk{\abs{\phit{n}{\theta}-\phit{n_k}{\theta}}<\e}\supset \gk{\Pt\rk{\abs{Y_n-\Ynk}>Y_n^{1-\delta}\vee n^{\beta-\delta} \Big| Y_n>0}<\e/4}\, .
    \end{equation}
\end{lemma}
\begin{proof}
Set $\delta' = \frac{1}{2A}$ (recall from \cref{lem:subsequence suff} that $n_k \sim k^A$). For each $n \geq 1$ take $k$ such that $n \in [n_{k-1}, n_{k}]$ and write
\begin{align}
    D_n&=\gk{\abs{Y_n-\Ynk}>Y_n^{1-\delta}\vee n^{\beta-\delta}}\, \cup \gk{\Ynk < n^{\beta+\delta'}}\, .
\end{align}
Now set $\gamma_n = \frac{\Pt\rk{Y_n>0}}{\Pt\rk{\Ynk>0}}$ and note that, $\bP$-almost surely, 
\begin{equation}\label{eqn:gamma bounde}
1 - \frac{\eps}{8} < \gamma_n^{-1} \leq 1 \leq \gamma_n < 1+\frac{\eps}{8}
\end{equation}
for all sufficiently large $n$ by \cref{thm:ConvergenceOfPartitionFunction}. 
Now note that if the event on the right-hand side of \eqref{eqn:event containment} occurs, then provided $n$ is sufficiently large (depending on $\epsilon$), it follows by \cref{CorollaryConvergenceCondtionedSubsequence} and \eqref{eqn:gamma bounde} that $\Pt \left( D_n | Y_n>0 \right) + \Pt \left( D_n | Y_{n_k}>0 \right) < \epsilon/2$ (on a $\bP$-almost sure event). Since the Laplace transform is bounded by $1$, it is therefore sufficient to show that
\begin{equation}\label{Equation6201e}
    \abs{\Et\ek{\ex^{-\theta n_k^{-\beta}Y_{n_k}}\1\gk{D_n^c}|Y_{n_k}>0}-\Et\ek{\ex^{-\theta n^{-\beta}Y_{n}}\1\gk{D_n^c}|Y_{n}>0}} \leq \e/2\, 
\end{equation}
in order to prove the lemma. To this end, note that on the event $D_n^c$ we can write
\begin{align}
    \abs{\nk^{-\beta}\Ynk-n^{-\beta}Y_n}\le \abs{\nk^{-\beta}-n^{-\beta}}Y_{\nk} + n^{-\beta}\abs{\Ynk-Y_n} \leq n^{-\beta}Y_n^{1-\delta} + \Ocal\rk{n^{-(\frac{1}{2A} \wedge \delta)}}.
\end{align}
Here the last line follows by our choice of $\delta'$ since $n_k^{-\beta}-n^{-\beta} = \Ocal\rk{n^{-\beta-\frac{1}{A}}}$. 
Note also that $\Pt\rk{\Ynk=0 \middle\vert Y_n>0} = 1 - \gamma_n^{-1}$. Combining the previous observations, we see that there exists $C<\infty$ such that we can bound the left-hand side of Equation \eqref{Equation6201e} from above by 
\begin{align}\label{eqn:laplace subtraction}
    \Et\ek{\ex^{-\theta n^{-\beta}Y_n}\abs{1-\gamma_n\ex^{C\theta (n^{-\beta}Y_n^{1-\delta} + n^{-(\frac{1}{2A} \wedge \delta)})}}\1\gk{D_n^c} \middle\vert Y_n>0} + \frac{\eps}{8}\, .
\end{align}
In the above expectation, we have that $\ex^{-\theta n^{-\beta}(Y_n-CY_n^{1-\delta})}=o(1)$ when $Y_n > n^{\frac{\beta}{1-\delta/2}}$ and also that $n^{-\beta}Y_n^{1-\delta}+n^{-(\frac{1}{2A} \wedge \delta)}=o(1)$ when $Y_n\le n^{\frac{\beta}{1-\delta/2}}$. We combine this with \eqref{eqn:gamma bounde} to bound \eqref{eqn:laplace subtraction} by
\begin{align}
&\Et\ek{\ex^{-\theta n^{-\beta}Y_n}\abs{1-\gamma_n + \gamma_n(1-\ex^{Cn^{-\beta}Y_n^{1-\delta} + Cn^{-(\frac{1}{2A} \wedge \delta)}})}\1\gk{D_n^c}\big|Y_n>0} + \frac{\eps}{8} < \frac{\eps}{4} + o(1),
\end{align}
thus completing the proof of \eqref{Equation6201e}.
\end{proof}

To ease notation, we fix $\delta \in (0,1)$ which will be made small in Proposition \ref{prop:verify condition}, in a way that depends only on $\alpha$ and other constants, and set, for $n_{k-1}< n\le n_k$,
\begin{align}\label{EquationEn}
 E_n=\gk{\abs{Y_n-\Ynk}>Y_n^{1-\delta}\vee n^{\beta-\delta}}.
\end{align} 
We now set $\eps>0$ and abbreviate the event $ F_n$ which depends on the tree up to generation $n$ (recall also the notation of \eqref{eqn:sigma alg conditioning}):
\begin{equation}
    F_n=F_n(\bT,\e)=\gk{\Pt\rk{E_n |\Gcal_n^\bT,Y_n>0}<\e/4}\, .
\end{equation}
The aforementioned criterion will be as follows: there exists $c>0$ such that for all $n$ large enough
\begin{equation}\label{Equation6211}
    \bP\rk{F_n|\Fcal_n}\ge c>0\, \qquad\text{everywhere}\, .
\end{equation}

We now prove that \eqref{Equation6211} is a sufficient condition for the final result.

\begin{lemma}[Criterion for full convergence]\label{lem:lift to whole sequence}
    If there exists a choice of $\delta>0$ such that for every $\eps>0$ there exists $c>0$ such that Equation \eqref{Equation6211} holds for all sufficiently large $n$, we have that $n^{-\beta}Y_n$, conditioned on survival up to time $n$, converges in distribution to the random variable with distribution function $\phi$.
\end{lemma}
\begin{proof}
Assume that $\delta \in (0, \frac{\beta}{2})$ is fixed. Now fix some $\eps>0$ and for all $n, k \geq 0$ write 
    \begin{align}
    \begin{split}
        A_n&=\gk{\abs{\phit{n}{\theta}-\phi(\theta)}>2\e}\, , \\
        B_k&=\gk{\abs{\phit{\nk}{\theta}-\phi(\theta)}>\e}\, ,
        \end{split}
    \end{align}
Note that $A_n$ is $\Fcal_{n}$ measurable and $B_k$ is $\Fcal_\nk$ measurable.
    
    By Lemma \ref{LemmaLink}, we have that there $\bP$-almost surely exists $K_o<\infty$ such that for all $k\ge K_o$ and all $n \in [n_{k-1}, n_{k}]$, $A_n$ implies $B_k$ on the event $F_n$. In other words we have for all $n_{k-1}\leq n\leq\nk$ that
    \begin{equation}\label{eqn:A F B inclusion}
        A_n\cap F_n\subset B_k\, .
    \end{equation}
  It follows that for any $K>K_o$,
    \begin{align}
    \begin{split}\label{eqn:AF B inclusion}
        \bP\rk{\exists k\ge K\colon B_k\text{ occurs}}&\ge \bP\rk{\exists n\ge n_K \colon A_n \cap F_n\text{ occurs}}.
   \end{split}
    \end{align}
Set $\tau_K = \inf \{ n \geq n_K: A_n \text{ occurs}\}$. Since $\tau_K$ is a stopping time, on the event $\{\tau_K < \infty\}$ we have that $\bP\rk{F_{\tau_K}|\tau_K < \infty, \Fcal_{\tau_K}} \ge c$ by \eqref{Equation6211}. Hence we can lower bound the right-hand side of \eqref{eqn:AF B inclusion} by 
\begin{align}\label{eqn:io from subsequence}
\bE\ek{    \bP\rk{\tau_K < \infty \text{ and } F_{\tau_K}|\Fcal_{\tau_K}}} \geq c\bP \rk{\tau_K < \infty}.
\end{align}
Combining \eqref{eqn:AF B inclusion} and \eqref{eqn:io from subsequence} gives
\begin{align}
   \bP\rk{\exists k\ge K\colon B_k\text{ occurs}} \geq c\bP\rk{\exists n\ge n_K \colon A_n\text{ occurs}}.
\end{align}
On account of Corollary \ref{CorollaryConvergenceCondtionedSubsequence}, we have that $B_k$ occurs only finitely many times $\bP$-almost surely, and hence the left hand side above tends to $0$ as $K \to \infty$. We deduce that the same holds for the right hand side, and therefore that $\bP \rk{A_n \text{ i.o.}}=0$. Since this holds for any $\eps>0$, this proves the claim.
\end{proof}

For the rest of the proof we therefore focus on establishing \eqref{Equation6211}. Let $\OpenEd_n=\gk{v\in \bT_n\colon 0\leftrightarrow\bT_n}$. To this end we fix $\eps>0$ and define a collection of ``good" sets of vertices at level $n$ in $\bT$ by setting
\begin{equation}
    A_\bT\hk{n}=\gk{x_n\subset \bT_n\colon \Pt\rk{E_n|\OpenEd_n=x_n}<\e/4}\, .
\end{equation}
Note that $\gk{\OpenEd_n=x_n}$ is measurable with respect to the $\sigma$-algebra generated by the set of possible percolation configurations up to level $n$ on $\bT$, which we recall is denoted $\Gcal_n^\bT$.

The next proposition contains the key estimate.

\begin{proposition}[All subsets are good]\label{prop:verify condition}
There is a choice of $\delta>0$ in \eqref{EquationEn} such that, for all $x_n\subset\bT_n$ we have that (uniformly over all choices of $x_n$ and all possibilities for $\bT_n$)
    \begin{equation}\label{EquationUniformXnBound}
        \bP\rk{x_n\in A_\bT\hk{n}|\Fcal_n}=1-o(1)\, .
    \end{equation}
\end{proposition}
\begin{proof}
First some notation: in this proof, for $a, b, c, d \in \R$ with $c<d$ we write $a=b[c,d]$ if there exists $C \in [c,d]$ such that $a=bC$.

   Now to the proof: fix $n \geq 0$ and assume that $n_{k-1} \leq n \leq n_k$. Take $x_n \subset \bT_n$, set $N=|x_n|$ and write $\gk{v_1,\ldots,v_N}$ for the (left to right) enumeration of the vertices contained in $x_n$. With respect to $\bP$, conditionally on $\Fcal_n$, we write $\bT\hk{i}_n$ for the subtree emanating from $v_i$ for $i=1,\ldots,N$. Note that $\rk{\bT^{(i)}_n}_{i=1}^N$ is a family of i.i.d. Galton--Watson trees. For each $1 \leq i \leq N$ and $k \geq 0$ also let $Y\hk{i}_k$ denote the number of vertices in the $k^{th}$ generation of $\bT\hk{i}_n$ that are connected to $v_i$ by an open path. We can then write, conditionally on the event $\{y_n=x_n\}$,
    \begin{equation}\label{eqn:Yn diff sum}
        \abs{Y_n-\Ynk}=\abs{\sum_{i=1}^N(Y\hk{i}_l-1)}\, ,
    \end{equation}
    where $l=\nk-n=\Ocal\rk{k^{A-1}}=\Ocal\rk{n^{\frac{A-1}{A}}}$.
    
    Recall from the introduction that $\P$ is the law of a critical Galton--Watson tree corresponding to the annealed percolation model. We therefore have (where $E_n$ is as in \eqref{EquationEn}) that
    \begin{equation}\label{eqn:expected deviation}
        \bE\ek{\Pt\rk{E_n|y_n=x_n}|\Fcal_n}=\P^{\otimes N}\rk{\abs{\sum_{i=1}^N(Y\hk{i}_l-1)}>N^{1-\delta}\vee n^{\beta-\delta}}\, ,
    \end{equation}
    where we abuse notation and use $Y\hk{i}_l$ as the number of surviving vertices in the annealed critical tree for the rest of this proof.
    
    Let $\kappa= \# \{i \leq N: Y_l\hk{i} \neq 0\}$, so that by \eqref{eqn:slack prob convergence} $\kappa$ is binomially distributed with parameters $p_l:=\P\rk{Y_l>0}\sim \CC l^{-\beta}$ and sample size $N$ under $\P$. We then have that
  \begin{multline}
        \P^{\otimes N}\rk{\abs{\sum_{i=1}^N(Y\hk{i}_l-1)}>N^{1-\delta}\vee n^{\beta-\delta}}\\
        =\sum_{\kappa=0}^N\binom{N}{\kappa}p_l^\kappa(1-p_l)^{N-\kappa}{\P^{\otimes \kappa}\rk{\abs{-(N-\kappa)+\sum_{i=1}^\kappa(\tilde{Y}\hk{i}_l-1)}>N^{1-\delta}\vee n^{\beta-\delta}}}\, ,
    \end{multline}
    where $\tilde{Y}\hk{i}_l$ are now the i.i.d. sizes of the populations conditioned to survive to level $l$. Note that $\E\ek{\tilde{Y}\hk{i}_l}=p_l^{-1}$ by Bayes' formula and since the original process was critical.

 Now pick some $\eps>0$ and for each $l \geq 1$ and $\kappa \geq 1$ define the event
    \begin{equation}
        A_l(\kappa)=\gk{\sum_{i=1}^\kappa \tilde{Y}_l\hk{i}=\kappa p^{-1}_l+\kappa^{1/\alpha+\eps}p_l^{-1}[-1,1]}.
    \end{equation}
    On the event $A_l(\kappa)$, we get that our sum $\sum_{i=1}^N ({Y}_l\hk{i}-1)$ is equal to $-N+\kappa p^{-1}_l+\kappa^{1/\alpha+\epsilon}p^{-1}_l[-1,1]$. Now for an additional parameter $M>0$ also consider the event 
    \begin{equation}\label{Equation12052023}
        B_{N,M}=\gk{\kappa=Np_l+\sqrt{Np_l}[-M,M]}\, .
    \end{equation}

   We henceforth fix some small $\eta>0$ and take $M=n^{\eps}$. By Lemma \ref{LemmaAnnealedLargeDeviation} (in the Appendix), we get that $A_l(\kappa)$ occurs with probability at least $1-\eta$, given $\kappa>\kappa_0$ for some $\kappa_0$ depending on $\alpha>1$ and $\eta$. We will treat the case $\kappa\le \kappa_0$ later. 
   
   Recall the Chebychev bound $\P\rk{\abs{X-np}\ge t\sqrt{np(1-p)}}\le t^{-2}$, for $X$ a binomial random variable with parameters $n$ and $p$. Since we are only interested in upper bounds, we bound $1-p_l$ by $1$. This, together with the previous paragraph shows that we can conclude that $B_{N,M}$ and $A_l(\kappa)$ occur with probability at least $1-\eta$, and we henceforth assume they occur.
   
Now note that there exists some $c=c(M)=\Ocal{M(Np_l)^{-1/2}}<\infty$ such that on $B_{N,M}$, we have that
    \begin{equation}
        \kappa^{1/\alpha+\eps}\le c\ek{Np_l \vee 1}^{1/\alpha+\eps}\, ,
    \end{equation}
and hence there exists another constant $c<\infty$ such that the sum in the equation below is bounded by (also using the fact that  $p_l \sim \CC l^{-\beta}$)
\begin{align}\label{Equ2010231}
    \begin{split}
       \sum_{i=1}^\kappa \tilde{Y}_l\hk{i} - N \leq \kappa p^{-1}_l+\kappa^{1/\alpha+\eps}p_l^{-1}[-1,1] - N &\leq \sqrt{N/p_l}[-M,M]+c\ek{Np_l \vee 1}^{1/\alpha+\eps}p_l^{-1} [-1,1] \\
       &\le   cN^{1/2}l^{\beta/2}[-M,M]+c(l^{\beta} \vee N^{1/\alpha+\eps}l^{\beta(1-1/\alpha - \eps)}[-1,1])\\
        &\le cN^{1/2}l^{\beta/2}[-M,M]+c(l^{\beta} \vee N^{1/\alpha+\eps}l^{1/\alpha}[-1,1])\, .
    \end{split}
\end{align}
Now take some $\xi>0$ and wlog assume that we originally picked $\eps>0$ small enough that $\eps < \xi/2\alpha$. Then note that if $l\le N^{-\xi}\min\gk{N^{1-2\delta}, N^{\frac{1-\delta-1/\alpha}{\beta(1-1/\alpha)}}}=N^{-\xi }N^{\alpha-\alpha\delta-1}$ (since $\delta<1$ and $\alpha=(\beta+1)/\beta$), we have that the equation above is at most $cn^{\eps}N^{1-\delta-\frac{\xi}{2\alpha}}$ and hence smaller than $N^{1-\delta} \vee n^{\beta-\delta}$ (for all sufficiently large $n$).
    
    On the other hand, assume now that $l>N^{-\xi }N^{\alpha-\alpha\delta-1}$. Write $B=(A-1)/A$ and set $\zeta=\frac{\alpha \delta + \xi}{\alpha-1}$ and $\zeta' = \frac{2\zeta}{\alpha}+ \eps$. Recall that there exists $C<\infty$, not depending on $n$, such that $l\le C n^{B}$ (see the remark under \eqref{eqn:Yn diff sum}) and recall that $l\ge N^{-\xi+\alpha(1-\delta)-1}=N^{(\alpha-1)(1-\zeta)}$. Provided that we chose $\delta$ and $\xi$ sufficiently small in the first place, we have that $N\le l^{\beta+2\zeta}$. This implies that the latter expression in \eqref{Equ2010231} is upper bounded by
    \begin{equation}\label{eqn:kappa N l calc}
cl^{\beta+\zeta}[-M,M]+cl^{(\beta+1)/\alpha+2\zeta/\alpha+\eps}[-1,1] = cl^{\beta+\zeta}[-M,M]+cl^{\beta+\zeta'}[-1,1] \leq  cn^{B(\beta+\zeta')+\eps} \, .
    \end{equation}
Here the final inequality holds on account of the condition imposed on $l$, and on substituting $M = n^{\eps}$.

Now note that $\zeta'>0$ can be made as small as we want, by decreasing $\delta, \eps$ and $\xi$ towards $0$. Hence, since $B<1$, we can decrease these all sufficiently that the right-hand side of \eqref{eqn:kappa N l calc} is at most $n^{\beta-\zeta}$.

It remains to show treat the case $\kappa\le \kappa_0$. However, in that case, we can assume that there is some $L=L(\eta,\kappa_0)$ such that $p_lN\le L$; otherwise $B_{N,M}$ doesn't hold, which can only happen with probability $1-\eta$, provided we chose $L$ large enough. This implies that for some $L'$, we have $N\le L' n^{\beta\rk{1-A^{-1}}}$. From there on, we can proceed as before, as $
\sum_{i=1}^{\kappa_0}\tilde{Y}_l\hk{i}<n^{\beta-\delta}$ outside a set of vanishing probability.

    It follows that, for all $\eta>0$ and $n$ sufficiently large
    \begin{equation}
        \bE\ek{\Pt\rk{E_n|\OpenEd_n=x_n}|\Fcal_n}\le \eta\, .
    \end{equation}
    Hence using Markov's inequality we deduce that
    \begin{equation}
        \bP\rk{x_n\notin A_\bT\hk{n}|\Fcal_n}\le \eta\, 
    \end{equation}
    for all sufficiently large $n$. As $\eta>0$ was arbitrary, the claim follows.
\end{proof}

\cref{thm:ConvergenceOfLaplaceTransform} now follows as a direct consequence of \cref{lem:lift to whole sequence} and \cref{sec:BrachningProcessProf}.

\section{Convergence of the branching process: proof of \cref{thm:ConvergenceOfBP}}\label{sec:BrachningProcessProf}
In this section we will consider the convergence of the process $(n^{-\beta}Y_{n(1+t)})_{t \geq 0}$ under the conditioning $Y_{n}>0$, in the Skorokhod-$J_1$ topology. Note that the Skorokhod-$J1$ topology is the the strongest amongst the different Skorokhod topologies, see \cite{SkorTop2}. For further background and definitions on the Skorokhod space of \cadlag functions (denoted $D[0, \infty)$) and the Skorokhod-$J_1$ topology we refer to \cite[Section 12]{billingsley1968convergence}. As outlined in the introduction, it is known that under the annealed law the process $(n^{-\beta}Y_{n(1+t)})_{t \geq 0}$ under the conditioning $Y_n>0$ converges to a \textit{continuous-state branching process} with entrance measure determined by \eqref{eqn:slack Laplace convergence}; see for example \cite[Theorem 2.1.9]{li2012continuousstate}. More precisely, the limiting process $(\tY_t)_{t \geq 0}$ is a Markov process on $[0,\infty)$ such that $\tY_0$ is a random variable with Laplace transform given by $\phi (\theta)$, and for all $t,x, y \geq 0$ its transition kernels $P_{t}(x,y)$ satisfy the \textit{branching property}
\begin{equation}\label{eqn:branching prop 2}
    P_{t}(x+y,\cdot) = P_{t}(x,\cdot) \ast P_{t}(y,\cdot).
\end{equation}
We will continue to denote these transition kernels by $P_t(\cdot, \cdot)$ throughout this section. We also remark that for $a>0$ and $t>s>0$, the law of $\tY_t$ conditionally on the event $\{\tY_s=a\}$ is that of a sum of $N$ independent copies of $X$ where $X$ has Laplace transform $\phi (\theta (t-s)^{\beta} )$ and $N$ has the distribution $\textsf{Poisson}(a \CC (t-s)^{-\beta})$. In particular its Laplace transform is therefore given by
\begin{equation}\label{eqn:annealed Poisson mgf}
Q_{a,s,t}(\theta) := \E \ek{e^{-\theta \tY_t} \vert \tY_s = a} = \E \ek{  \E \ek{e^{-\theta X}}^N} = \exp\{ -a \CC(t-s)^{-\beta} (1-\phi (\theta (t-s)^{\beta}))\}.
\end{equation}
In fact we can and will use \eqref{eqn:annealed Poisson mgf} as a characterisation of this limiting branching process.

By \cite[Theorem 13.1]{billingsley1968convergence}, we can prove \cref{thm:ConvergenceOfBP} by first showing that the sequence of processes is tight, and then showing that the finite dimensional marginals converge for all $t\in [0,\infty)$.

\subsection{Tightness}
We use the standard criteria for tightness in the Skorokhod-$J_1$ topology given in \cite[Theorem 13.2]{billingsley1968convergence}. If $X=(X_t)_{t \geq 0}$ is a \cadlag function, we define 
\[
||X||_{\infty} = \sup_{t \geq 0} |X_t|, \qquad \omega_\delta '(X) := \inf_{\{t_i\}_{i=0}^m} \sup_{1 \leq i \leq m} \sup_{s, t \in [t_{i-1}t_i)} |X_t - X_s|,
\]
where the infimum is taken over all partitions $0=t_0 < t_1 < \ldots < t_m$ with $t_i - t_{i-1} > \delta$ for all $i$. For $K\geq 0$, we also let $X^K$ denote the function defined by 
\[
X^K_t = X_t \mathbbm{1}\{t \leq K\}.
\]

We will use the following result.

\begin{proposition}\cite[Theorem 13.2]{billingsley1968convergence}.\label{prop:billingsley tightness criteria}
    Let $(X^{(n)})_{n \geq 0}$ be a sequence of \cadlag functions $[0, \infty) \to \R$ defined on a probability space $(\Omega, \Fcal, P)$. For $K>0$ set $X^{(n,K)} = (X^{(n)})^K$. Then the sequence is tight if and only if for every $K>0$, the following conditions are satisfied.
    \begin{enumerate}[(a)]
        \item $\lim_{C \to \infty} \limsup_{n \to \infty} P \rk{||X^{(n,K)}||_{\infty} \geq C} = 0$.
        \item For all $\eps>0$, $\lim_{\delta \downarrow 0} \limsup_{n \to \infty} P \rk{ \omega_{\delta}'(X^{(n,K)}) > \eps} = 0$.
    \end{enumerate}
\end{proposition}

This leads us rather straightforwardly to the following proposition.

\begin{proposition}\label{prop:tightness}
For each $n \geq 1$ let $(\tY^{(n)}_t)_{t \geq 0}$ denote the process $(n^{-\beta}Y_{n(1+t)})_{t \geq 0}$ under the conditioning $Y_{n}>0$ (i.e., set $\tY^{(n)}_t = n^{-\beta}Y_{n(1+t)}$ under this conditioning). Then, $\bP$-almost surely, the sequence of processes $(\tY^{(n)})_{n \geq 1}$ is tight with respect to the Skorokhod-$J_1$ topology.
\end{proposition}
\begin{proof}
    If $A_K$ is the subset of $\Omega$ where either (a) or (b) of \cref{prop:billingsley tightness criteria} fail for $K$, then we have that $A_K$ is an increasing function of $K$. As a result, in order to prove this proposition it is sufficient to show that, for every $K>0$, the following hold $\bP$-almost surely.
    \begin{enumerate}[(a)]
        \item $\lim_{C \to \infty} \limsup_{n \to \infty} \Pt \rk{||\tY^{(n,K)}||_{\infty} \geq C} = 0$.
        \item For all $\eps>0$, $\lim_{\delta \downarrow 0} \limsup_{n \to \infty} \Pt \rk{ \omega_{\delta}'(\tY^{(n,K)}) > \eps} = 0$.
    \end{enumerate}

    We claim that, thanks to the convergence of the annealed process, the following two statements hold.
       \begin{enumerate}[(a$'$)]
        \item $\bE \ek{\lim_{C \to \infty} \limsup_{n \to \infty} \Pt \rk{||\tY^{(n,K)}||_{\infty} \geq C}} = 0$.
        \item For all $\eps>0$, $\bE \ek{\lim_{\delta \downarrow 0} \limsup_{n \to \infty}  \Pt \rk{ \omega_{\delta}'(\tY^{(n,K)}) > \eps}} = 0$.
    \end{enumerate}
    Note that the expectation of a sequence of non-negative random variables can only be zero if the random variables in question are zero almost surely. Hence (a$'$) and (b$'$) directly imply that (a) and (b) are satisfied, $\bP$-almost surely, and it just remains to establish (a$'$) and (b$'$).
    
    To this end, set $x_C = \limsup_{n \to \infty} \Pt \rk{||\tY^{(n,K)}||_{\infty} \geq C}$. Suppose for a contradiction that the expectation in (a$'$) is not $0$, and is in fact lower bounded by $2c>0$. Then $\bP \rk{\lim_{C \to \infty} x_C \geq c} \geq c$.
    
   Moreover, on the event $\{x_C > c\}$, we have by Fatou's lemma that $\Pt \rk{||\tY^{(n,K)}||_{\infty} \geq C \text{ i.o.}} \geq c$. Since this holds for all $C$ on the event in question (and the event is monotone in $C$), it follows from monotone convergence that $\Pt \rk{\forall C>0 : ||\tY^{(n,K)}||_{\infty} \geq C \text{ i.o.}} \geq c$. However, on the event $\{\forall C>0 : ||\tY^{(n,K)}||_{\infty} \geq C \text{ i.o.}\}$ it follows from \cite[Theorem 12.3]{billingsley1968convergence} that $\tY^{(n,K)}$ cannot belong to a compact set. Since this event occurs with $\P$-probability at least $c^2$, this would imply that the annealed law is not tight, which is a contradiction. 

    Similarly, suppose for a contradiction that there exists $\eps>0$ such that the expectation in (b$'$) is lower bounded by $2c$. Set $y_{\delta}:= \limsup_{n \to \infty}  \Pt \rk{ \omega_{\delta}'(\tY^{(n,K)}) > \eps}$. Then necessarily $\bP \rk{\lim_{\delta \to 0} y_{\delta}>c}>c$, and moreover on the event $\{y_{\delta} > c\}$, we have by Fatou's lemma that $\Pt \rk{ \omega_{\delta}'(\tY^{(n,K)}) > \eps \text{ i.o.}} >c$.\\
    By monotone convergence, it follows that $\Pt \rk{\forall \delta>0: \omega_{\delta}'(\tY^{(n,K)}) > \eps \text{ i.o.}} >c$, and moreover on the event $\{\forall \delta>0: \omega_{\delta}'(\tY^{(n,K)}) > \eps \text{ i.o.}\} $, it follows from \cite[Theorem 12.3]{billingsley1968convergence} that $\tY^{(n,K)}$ cannot belong to a compact set. Since this event occurs with $\P$-probability at least $c^2$, this would similarly imply that the annealed law is not tight, which is a contradiction. 
\end{proof}

\subsection{Convergence of the finite dimensional distributions along a subsequence}
Throughout this section, we abbreviate $Y_t\hk{n}=n^{-\beta}Y_{\floor{tn}}$. As in the previous subsections, the strategy will be to show that the finite dimensional marginals converge almost surely along a subsequence, and then lift this to convergence along the full sequence using the continuity event $F_n$ that we used in \cref{lem:lift to whole sequence}.

We start with a variance bound that will be useful for some technical estimates.

\begin{lemma}[A second moment estimate]\label{lem:variance bound}
    For any $\e>0$, as $n \to \infty$
    \begin{equation}
        \bE\ek{\Pt\rk{Y_n>0}^2}=\Ocal\rk{n^{-\beta\alpha+\e}}\, .
    \end{equation}
\end{lemma}
\begin{proof}
Fix $\delta>0$ and write $p=\alpha-\delta$. We bound the second moment by the $p^{th}$ moment:
\begin{equation}\label{eqn:variance alpha moment comp}
     \bE\ek{\Pt\rk{Y_n>0}^2}\le \bE\ek{\Pt\rk{Y_n>0}^p}\, .
\end{equation}
We now set $m = \lceil \frac{\beta}{\log \mu} \log n \rceil$. Note that, by a union bound, we have that
\begin{multline}
    \Pt\rk{Y_n>0}\le \sum_{v\in\bT_m}\mu^{-m}\Pt\rk{v\lrd \bT_n}\\
    =\sum_{v\in\bT_m}\mu^{-m}\P\rk{\root\lrd \bT_{n-m}}+\sum_{v\in\bT_m}\mu^{-m}\ek{\Pt\rk{v\lrd \bT_n}-\P\rk{\root\lrd \bT_{n-m}}}:=X_1+X_2\, ,
\end{multline}
where $X_1$ and $X_2$ denote the first and second sums respectively. Now note that
\begin{equation}
    \bE\ek{X_1^p}\le \Ocal\rk{(n-m)^{-\beta p}}\bE\ek{\bW_m^p}=\Ocal\rk{n^{-\beta p}}, 
\end{equation}
by the tower property, \eqref{eqn:slack prob convergence} and since $\rk{\bW_m}_m$ converges in $\mathrm{L}^p(\bP)$ when $p< \alpha$.

Furthermore, by Equation \eqref{eqn:Cheb BC 1} and the tower property, we have that there exists $C<\infty$ such that for any $x>0$ (note that although we had set $m= \lfloor  \frac{(1+\e)}{\rk{\alpha-1}\log\mu}\log n \rfloor$ in the proof of \cref{lem:variance}, the proof up until \eqref{eqn:Cheb BC 1} did not use this, so it also holds for the present value of $m$)
\begin{equation}
    \bP\rk{|X_2|>x}\le C\frac{n^{-\beta}}{x^2 \mu^m}\, .
\end{equation}
In particular, we have for any $x>0$ that
\begin{equation}
    \bP\rk{|X_2|>xn^{-\beta}}\le C\frac{n^{\beta}}{x^2 \mu^m} \leq  \frac{C}{x^2}.
\end{equation}
Here the final inequality follows by our choice of $m$. Hence, the $p$-moment of $|X_2|$ exists and is bounded by $\Ocal\rk{n^{-\beta p }}$. To tie up, note that
\begin{align}
    \bE\ek{\Pt\rk{Y_n>0}^p} = \bE\ek{(X_1+X_2)^p} \leq 2^{p}(\bE\ek{X_1^p} + \bE\ek{|X_2|^p} ) = \Ocal\rk{n^{-\beta p }}= \Ocal\rk{n^{-\beta \alpha+\e }},
\end{align}
for some $\e>0$ depending on $\delta$, which can be as small as we want to. Substituting back into \eqref{eqn:variance alpha moment comp} concludes the proof.
\end{proof}

\begin{remark}
    We note that the above lemma allows for a subsequent improvement of the estimate made in Equation \eqref{eqn:first variance bound}. 
\end{remark}

We are now ready to prove convergence of the marginals along an appropriate subsequence. Recall the definition of the subsequence $(n_k)_{k \geq 1}$ that we defined in \cref{sctn:notation} (in fact it will only be important in this section that $n_k$ grow slower than $k^2$).

The convergence follows from \cref{lem:variance bound} and a simple application of Chebyshev's inequality. Before giving the proof, we remind the reader of the following formula for the variance of a product of independent variables $X_1,\ldots,X_M$:
    \begin{equation}\label{eqn:variance formula}
\Var{\prod_{i=1}^MX_i}=\prod_{i=1}^M\rk{\Var{X_i}+\bE\ek{X_i}^2}-\prod_{i=1}^M\bE\ek{X_i}^2\, .
    \end{equation}

\begin{lemma}[Pointwise convergence of the Laplace transform]\label{lem:mgf starting from a convergence 1}
Take any $a>0, t>0, s>0$, and for each $n \geq 1$ set $a_n = n^{-\beta} \lfloor an^{\beta} \rfloor$. Then following holds for $\bP$-almost every tree $\bT$. For each $n \geq 1$ set
\[
Q^{n,\bT}_{a,s,t}(\theta) = \Et \ek{\exp\{-\theta Y_{t}\hk{n}\} \Big| Y_s\hk{n}=a_n }
\]
and let $Q_{a,s,t}(\theta)$ denote the analogous quantity in the limit of the annealed model, i.e. as in \eqref{eqn:annealed Poisson mgf}. Then, for all $\theta>0$,
\begin{equation}
    Q^{n_k,\bT}_{a_n,s,t}(\theta) \to Q_{a,s,t}(\theta)
\end{equation}
as $k \to \infty$.
\end{lemma}
\begin{proof}
In the following proof we will abuse notation slightly by writing $t-s$ in place of $\frac{1}{n}\left( \lfloor t n \rfloor - \lfloor sn \rfloor \right)$, and $(t-s)n$ in place of $\lfloor t n \rfloor - \lfloor sn \rfloor$ (any resulting errors can be absorbed into the $(1+o(1))$ factors that are already present in the proof).

Fix some $s, t, a, \theta>0$. We condition on $\bT_{\lfloor sn \rfloor}$ and write
\[
S = \sum_{v \in \bT_{\lfloor sn \rfloor} \colon 0\lrd v}Y\hk{v}_{(t-s)n},
\]
where for $v \in \bT_{\lfloor sn \rfloor}$ and $\ell \in \N$ we write $Y\hk{v}_{\ell}$ to denote vertices in generation $\ell$ in the subtree rooted at $v$.

For each $n \geq 1$ recall that $\Fcal_n$ is the $\sigma$-algebra generated by possible realisations of the first $n$ levels of $\bT$, so that for the set $\OpenEd_{\lfloor sn \rfloor} := \{v \in \bT_{\lfloor sn \rfloor} \colon 0\lrd v\}$, the probability $\Pt\rk{\OpenEd_{\lfloor sn \rfloor}=A|\Ypron{s}=a_n}$ is $\Fcal_{\floor{sn}}$-measurable, for any $A\subset \bT_{\lfloor sn \rfloor}$. 
We now compute the moment-generating function $\Et\ek{\ex^{-\theta S}|\Ypron{s}=a_n}$ conditionally on $\Fcal_{\floor{sn}}$.

Write $X_n=\Et\ek{\ex^{-\theta S}|\Ypron{s}=a_n}$. Note that $\bE\ek{X_n|\Fcal_{\floor{sn}}}$ is given by
    \begin{equation}
       \ek{1-\P\rk{\Ypron{t-s}>0}\rk{1-\E\ek{\ex^{-\theta\Ypron{t-s}}|\Ypron{t-s}>0}}}^{a_nn^\beta}=\ex^{a (t-s)^{-\beta}\CC\rk{\phi(\theta  (t-s)^{\beta})-1}}\rk{1+o(1)}\, ,
    \end{equation}
    where the final equality holds by \eqref{eqn:slack Laplace convergence}. On the other hand, we have from \eqref{eqn:variance formula} that
    \begin{equation}\label{eqn:variance formula spec}
        \Var{X_n|\Fcal_{\floor{sn}}}=\ek{\Var{\Et\ek{\ex^{-\theta\Ypron{t-s}}}}+\E\ek{\ex^{-\theta\Ypron{t-s}}}^2}^{a_nn^\beta}-\E\ek{\ex^{-\theta\Ypron{t-s}}}^{2a_nn^\beta}\, .
    \end{equation}
Now take some small $\epsilon > 0$. We note that, by \cref{lem:variance bound}, there exists $C<\infty$ such that for all $n \geq 1$,
    \begin{align}\label{eqn:variance Laplace}
    \begin{split}
        \Var{\Et\ek{\ex^{-\theta\Ypron{t-s}}}}&=\Var{1-\Et\ek{\ex^{-\theta\Ypron{t-s}}}}\\
        &=\Var{\Pt\rk{\Ypron{t-s}>0}\left(1-\Et\ek{\ex^{-\theta\Ypron{t-s}}|\Ypron{t-s}>0}\right)}
        \\
        &\le \bE\ek{\Pt\rk{\Ypron{t-s}>0}^2} \le C n^{-\beta\alpha+\e}\, .
    \end{split}
    \end{align}
    Moreover, by \eqref{eqn:slack prob convergence} and \eqref{eqn:slack Laplace convergence}, we have that
    \begin{align}\label{eqn:exp conv}
     \E\ek{\ex^{-\theta \Ypron{t-s} }} = 1 - \CC (t-s)^{-\beta} n^{-\beta}(1-\phi (\theta (t-s)^{\beta})) (1+o(1)).
    \end{align}
Now recall that for $x, y >0$ and $m \geq 1$ we have that $(x+y)^m-y^m \leq mx(x+y)^{m-1}$. Applying this to \eqref{eqn:variance formula spec} and substituting the estimates from \eqref{eqn:variance Laplace} and \eqref{eqn:exp conv}, we deduce that there exist $c'>0$ and other $c, C, C'<\infty$, possibly depending on $\alpha, \theta, s$ and $t$ such that
    \begin{align}
    \begin{split}
        \Var{X_n|\Fcal_{\floor{sn}}} &\le a_nn^{\beta} \Var{\Et\ek{\ex^{-\theta\Ypron{t-s}}}} \ek{\Var{\Et\ek{\ex^{-\theta\Ypron{t-s}}}}+\E\ek{\ex^{-\theta\Ypron{t-s}}}^2}^{a_nn^\beta - 1}
        \\
        &\leq C n^{-\beta(\alpha-1)+\e} (1+cn^{-\alpha \beta+\epsilon} - c'n^{-\beta})^{a_n n^{\beta}-1} \\
        &\leq C' n^{-\beta(\alpha-1)+\e} = C' n^{-1+\e}.
    \end{split}
    \end{align}
        Hence Chebyshev's inequality shows that $X_{n} \to \ex^{a (t-s)^{-\beta}\CC\rk{\phi(\theta  (t-s)^{\beta})-1}}=Q_{a,s,t}(\theta)$ almost surely along the subsequence $(\nk)_{k \geq 1}$
        
        This currently proves the result for any fixed $\theta$, $\bP$-almost surely. Note that we can extend this to hold simultaneously for all $\theta>0$ using the same argument as in \cref{CorollaryConvergenceCondtionedSubsequence}.
\end{proof}

In fact we need to strengthen \cref{lem:mgf starting from a convergence 1} to the following.

\begin{lemma}[Subsequential convergence of the Laplace transform]\label{lem:mgf starting from a convergence}
 For $\bP$-almost every tree $\bT$ the following holds for almost every $a>0$ and every $t>s>0, \theta>0$. For each $n \geq 1$ again set $a_n = n^{-\beta} \lfloor an^{\beta} \rfloor$ and
\[
Q^{n,\bT}_{a,s,t}(\theta) = \Et \ek{\exp\{-\theta Y_{t}\hk{n}\} \Big| Y_s\hk{n}=a_n  }.
\]
Let $Q_{a,s,t}(\theta)$ denote the analogous quantity in the annealed model, i.e. as in \eqref{eqn:annealed Poisson mgf}. Then
\begin{equation}
    Q^{n_k,\bT}_{a,s,t}(\theta) \to Q_{a,s,t}(\theta)
\end{equation}
as $k \to \infty$.
\end{lemma}
\begin{proof}
In the previous lemma, we showed that for any fixed $s,t$ and $a$, the convergence holds for all $\theta>0$ for $\bP$-almost every tree. This immediately implies the following: for $\bP$-almost every tree, the statement is true for Lebesgue almost every $a, s, t\ge 0$, including all rationals, using Fubini's theorem. 

Now note that, $\bP$-almost surely, for any fixed $\theta$, any $s>0$ and almost every $a>0$, the function $Q^{n,\bT}_{a_n,s,t}(\theta)$ is equal to the limit of $Q^{n,\bT}_{a_n,s,t'}(\theta)$ as $t' \downarrow t$, and this continues to hold as $n\to\infty$. Otherwise, it would contradict point (ii) in \cref{prop:tightness}, which says that for all $\eps>0$, $\lim_{\delta \downarrow 0} \limsup_{n \to \infty} \Pt \rk{ \omega_{\delta}'(\tY^{(n,K)}) > \eps} = 0$ (here, we take $K>t$ w.l.o.g.). Since the limiting function $Q_{a,s,t}(\theta)$ is continuous in all variables, we deduce the following: for almost every $s$, it holds for almost every choice of $a$ that the function $Q_{a,s,t}(\theta)$ converges for all $\theta$ and all $t$. Denote the set of good $s$ where this holds by $S^{\text{good}}$.

Then if $s \notin S^{\text{good}}$, choose some other $s' \in S^{\text{good}}$ and set $c=\frac{s}{s'}$. Given $t>s$, also set $t' = \frac{t}{c}$. Then $Y^{(n)}_s = c^{\beta}Y^{(cn)}_{s'}$ and $Y^{(n)}_t = c^{\beta}Y^{(cn)}_{t'}$. Suppose moreover that $a$ is such that $ Q^{cn,\bT}_{ac^{-\beta},s',t'}(\theta)$ converges for all $\theta$ (this will be the case for almost every $a$ by virtue of the fact that $s' \in S^{\text{good}}$ and $c$ is fixed). It then follows that (also using \eqref{eqn:branching prop 2})
\begin{align}
    \begin{split}
 Q^{n,\bT}_{a,s,t}(\theta) = \Et \ek{\exp\{-\theta Y_{t}\hk{n}\} \Big| Y_s\hk{n}=a_n  }  &= \Et \ek{\exp\{-c^{\beta} \theta Y_{t'}\hk{cn}\} \Big| Y_{s'}\hk{cn}=a_nc^{-\beta}  } \\
 &=  Q^{cn,\bT}_{ac^{-\beta},s',t'}(c^{\beta} \theta)  \\
 &\to \exp \{-c^{-\beta}a\CC(t'-s')^{-\beta}(1-\phi(c^{\beta} \theta  (t'-s')^{\beta})\} \\
 &= \exp \{-a\CC(t-s)^{-\beta}(1-\phi( (t-s)^{\beta} \theta))\} \\
 &= Q_{a,s,t}(\theta).
    \end{split}
\end{align}
We therefore now have the following statement for $\bP$-almost every tree: for all $s$, we have for almost every $a$ that the convergence holds for all $t, \theta$.
\end{proof}

\begin{remark}
Similar results that could be applied to lift the convergence of \cref{lem:mgf starting from a convergence 1} to that of \cref{lem:mgf starting from a convergence} appear in \cite[Theorem 12.12]{balan2019weak}.
\end{remark}

\subsection{Lifting to the full sequence}
Similarly to \cref{sctn:Laplace transform proof}, we fix a small $\delta>0$ and define for every $\eps>0$ the events
\begin{align}\label{EquationEnn}
\begin{split}
    E_n&=\gk{\abs{Y_n-\Ynk}>Y_n^{1-\delta}\vee n^{\beta-\delta}}, \\
    F_n&=F_n(\bT,\e)=\gk{\Pt\rk{E_n |\Gcal_n^\bT, Y_n>0}<\e/4},
\end{split}
\end{align}
where $k$ is chosen so that $n_{k-1}\leq n\leq\nk$. We then moreover showed in \cref{prop:verify condition} that for all $\eps>0$,
\begin{equation}\label{Equation62111}
    \bP\rk{F_n|\Fcal_n}\ge 1-o(1)\, \qquad\text{everywhere}\, .
\end{equation}

For this proof we slightly modify the events in question: assume $sn,tn\in\N$ and set
\begin{align}\label{EquationEnnn}
\begin{split}
  F_n(a,s,t) =  \gk{\Pt\rk{E_{tn} |\Gcal_{tn}^\bT, Y_{tn}>0, Y_{sn}=\lfloor a n^{\beta} \rfloor }<\e/4}\, .
\end{split}
\end{align}
\begin{lemma}\label{lem:new criterion}
    For all $a> 0, t>s > 0$, as $n \to \infty$, 
  \begin{equation}
         \bP\rk{F_n(a,s,t)|\Fcal_{tn}} \geq 1 - o(1) .
    \end{equation} 
\end{lemma}
\begin{proof}
Since the event $\{Y_{sn}=\lfloor a n^{\beta} \rfloor\}$ is $\Gcal_{tn}^\bT$-measurable, this is a direct consequence of \cref{prop:verify condition}.
\end{proof}

We now prove that \cref{lem:new criterion} is a sufficient condition to lift \cref{lem:mgf starting from a convergence} to the whole sequence.

\begin{lemma}[Criterion for full convergence]\label{lem:lift to whole sequence 2}
Suppose that for every $\eps>0$ Equation \eqref{lem:new criterion} holds. Then for $\bP$-almost every tree $\bT$ the following holds for every $t>0, s>0, \theta>0$, for almost every $a>0$. For all $n \geq 1$ set $a_n = n^{-\beta} \lfloor an^{\beta} \rfloor$ and
\[
Q^{n,\bT}_{a,s,t}(\theta) = \Et \ek{\exp\{-\theta Y_{t}\hk{n}\} \Big| Y_s\hk{n}=a_n  },
\]
and let $Q_{a,s,t}(\theta)$ denote the analogous quantity in the annealed model, i.e. as in \eqref{eqn:annealed Poisson mgf}. Then
\begin{equation}
    Q^{n,\bT}_{a,s,t}(\theta) \to Q_{a,s,t}(\theta)
\end{equation}
as $n \to \infty$. 
\end{lemma}
\begin{proof}
The proof is similar to that of \cref{lem:lift to whole sequence} and we just write the details for completeness. In particular we already showed that the result holds almost surely along the subsequence $(n_k)_{k \geq 1}$ in \cref{lem:mgf starting from a convergence}.

Now fix some $\eps>0$ and for $n, k \geq 0$ write 
    \begin{align}
    \begin{split}
        A_n&=\gk{\abs{Q^{n,\bT}_{a,s,t}(\theta)-Q_{a,s,t}(\theta)}>2\e}\, , \\
        B_k&=\gk{\abs{Q^{\nk,\bT}_{a,s,t}(\theta)-Q_{a,s,t}(\theta)}>\e}\, .
        \end{split}
    \end{align}
    Note that $A_n$ is $\Fcal_{tn}$ measurable and $B_k$ is $\Fcal_{t\nk}$ measurable. For the rest of the proof we treat $a,s$ and $t$ as fixed and write $F_n$ in place of $F_n(a,s,t)$.
    
    By the same arguments as in Lemma \ref{LemmaLink}, we have that there exists $K_o<\infty$ such that for all $k\ge K_o$ and all $n \in [n_{k-1}, n_{k}]$, $A_n$ implies $B_k$ on the event $F_{n}$. In other words we have for all $n_{k-1}\leq n\leq\nk$ that
    \begin{equation}
        A_n\cap F_{n}\subset B_k\, .
    \end{equation}
  It follows that for any $K>K_o$,
    \begin{align}
    \begin{split}\label{eqn:AF B inclusion 2}
        \bP\rk{\exists k\ge K\colon B_k\text{ occurs}}&\ge \bP\rk{\exists n\ge n_K \colon A_n \cap F_{n} \text{ occurs}}.
   \end{split}
    \end{align}
Set $\tau_K = \inf \{ n \geq n_K: A_n \text{ occurs}\}$. Since $\tau_K$ is a stopping time, on the event $\{\tau_K < \infty\}$ we have that $\bP\rk{F_{\tau_K}|\tau_K < \infty, \Fcal_{t\tau_K}} \ge \frac{1}{2}$ (provided $K$ was large enough) by \cref{lem:new criterion}. Hence we can lower bound the right-hand side of \eqref{eqn:AF B inclusion 2} by 
\begin{align}\label{eqn:io from subsequence 2}
\bE\ek{    \bP\rk{\tau_K < \infty \text{ and } F_{t\tau_K}|\Fcal_{\tau_K}}} \geq \frac{1}{2}\bP \rk{\tau_K < \infty}.
\end{align}
Combining \eqref{eqn:AF B inclusion 2} and \eqref{eqn:io from subsequence 2} gives
\begin{align}
   \bP\rk{\exists k\ge K\colon B_k\text{ occurs}} \geq \frac{1}{2}\bP\rk{\exists n\ge n_K \colon A_n\text{ occurs}}.
\end{align}
On account of \cref{lem:mgf starting from a convergence}, we have that $B_k$ occurs only finitely many times $\bP$-almost surely, and hence the left hand side above tends to $0$ as $K \to \infty$. We deduce that the same holds for the right hand side, and therefore that $\bP \rk{A_n \text{ i.o.}}=0$. Since this holds for any $\eps>0$, this proves the claim.
\end{proof}

\begin{cor}
        We have that $\bP$-almost surely, for all $t>s>0$ it holds for almost every $a>0$ and any Borel set $B$ that
    \begin{equation}
       \Pt(Y_{t}\hk{n}\in B|Y_s\hk{n}=a) = P_{t-s}(a,B) + o(1).
    \end{equation}
\end{cor}
\begin{proof}
\cref{lem:lift to whole sequence 2} implies that for all $t>s>0$ that for almost every $a>0$, we have almost sure convergence in distribution of the law $\Pt\rk{Y_{t}\hk{n} \in \cdot|Y_s\hk{n}=a}$ to $\P(Y_{t} \in \cdot|Y_s=a)$.
\end{proof}

By iterating this result and combining with \cref{thm:ConvergenceOfLaplaceTransform}, we deduce the following.

\begin{lemma}\label{lem:fd convergence}
For $\bP$-almost every tree $\bT$ the following holds. Take any $d \geq 1$, any $t_1 < t_2 < \ldots < t_d \in (1, \infty)$ and let $f$ be the density function of the random variable with Laplace transform $\phi( \theta)$ as in \eqref{eqn:slack Laplace convergence}. Take any Borel sets $B_1,\ldots,B_d\subset [0,\infty)$. Also set $t_0=1$ and let $P_t$ be as in \eqref{eqn:branching prop 2}. Then
\begin{multline}
    \lim_{n\to\infty}\Pt\rk{n^{-\beta}Y_{\floor{t_in}}\in B_{i}\text{ for all }i\in\gk{1,\ldots,d}\big|Y_n > 0}\\
    =\int_{B_0}\d x_0 \int_{B_1}\d x_1\ldots\int _{B_d}\d x_d f(x_0)\prod_{i=1}^d P_{t_i-t_{i-1}}(x_{i-1},x_i) \, . 
\end{multline}
\end{lemma}

In particular, \cref{prop:tightness} and \cref{lem:fd convergence} together imply the result of \cref{thm:ConvergenceOfBP}.

\section{The IIC: proof of \cref{thm:IIC size biased}}\label{sec:IICProof}
Let us first define the measure on the IIC. For this, for a tree $T$ we let $T[n]$ denote the subtree of $T$ obtained by keeping only the vertices at level $n$ and below, and the edges between them. Recall also that for a finite tree $t$, $t_n$ denotes the collection of vertices at exactly level $n$ in $t$. If $\root \subset t \subset \bT$, the height of $t$ is equal to the maximal generation reached by $t$ (i.e. $\sup_{v \in t}|v|$).

In this section we let $\Hcal_n$ denote the set of subtrees of $\bT$ that contain $\root$ and have height $n$.

\begin{definition}[Law of the IIC]
    For any tree $t \in \Hcal_n$ with height $n$, set 
    \begin{equation}\label{eqn:IIC def}
        \mu_\bT\Big|_{n}[t]=\frac{\sum_{v\in t_n}\bW(v)}{\bW}\Pt\rk{\Ccal_{p_c}[n]=t}\, ,
    \end{equation}
    where $\Ccal_{p_c}$ is the root cluster, and hence $\Ccal_{p_c}[n]$ is the root cluster restricted to levels $n$ and below.
\end{definition}

We see that this is well-defined in the following lemma.

\begin{lemma}\label{lem:IIC well defined}
   For every $t \in \Hcal_n$,
\begin{equation}
    \mu_\bT\Big|_{n}[t]=\lim_{M\to\infty}\Pt\rk{\Ccal_{p_c}[n]=t\big |0\leftrightarrow\bT_{M}}\, .
\end{equation}
Moreover, the collection of measures $\left(\mu_\bT\Big|_{n}\right)_{n \geq 1}$ are consistent and therefore extend to a unique measure on infinite trees, which we will henceforth denote by $\mu_{\bT}$.
\end{lemma}
\begin{proof}
The first part of the proof is a carbon copy of that of \cite[Lemma 3.9]{michelen2020quenched}. In particular, we can write (applying \cref{thm:ConvergenceOfPartitionFunction} in the second and third steps, and where $O(\cdot)$ refers to the limit as $M \to \infty$ with $n$ fixed))
\begin{align}
    \begin{split}
 \Pt\rk{\Ccal_{p_c}[n]=t\big |0\leftrightarrow\bT_{M}} &= \frac{\Pt\rk{\Ccal_{p_c}[n]=t  \text{ and } 0\leftrightarrow\bT_{M}}}{\Pt\rk{0\leftrightarrow\bT_{M}}} \\
 &= \Pt\rk{\Ccal_{p_c}[n]=t} \frac{\sum_{v \in \bT_n} \Pt\rk{v\leftrightarrow\bT_{M}} + O(|t_n|^2M^{-2\beta})}{\Pt\rk{0\leftrightarrow\bT_{M}}} \\
 &\to \frac{\sum_{v\in \bT_n}\bW(v)}{\bW}\Pt\rk{\Ccal_{p_c}[n]=t}.
    \end{split}
\end{align}
To prove that the measures are consistent, we let $\textsf{Ch}(t_n)$ denote the set of children of $t_n$, and let $\textsf{Ch}_k(t_n)$ denote the collection of subsets of $\textsf{Ch}(t_n)$ of size $k$. Note that, if $v$ is a child of a vertex in $t_n$, then
\[
|\{S \subset \textsf{Ch}_k(t_n): v \in S\}|= \binom{|\textsf{Ch}(t_n)|-1}{k-1}\, .
\]
Hence, if $t \in \Hcal_n$, we can write (note that going from the third to the fourth line we can eliminate the option $k=0$, since the corresponding tree has probability $0$ under $\mu_\bT\Big|_{{n+1}}$ by \cref{eqn:IIC def})
\begin{align}
    \begin{split}
        \mu_\bT\Big|_{{n+1}}\left[t\right] &=\mu_\bT\Big|_{{n+1}}\left[\{t' \in \Hcal_{n+1} : t'[n]=t\}\right] \\
        &= \sum_{S \subset \textsf{Ch}(t_n)} \frac{\sum_{v\in S}\bW(v)}{\bW}\Pt\rk{\Ccal_{p_c}[n+1]=t \cup \{S\}} \\
        &=\sum_{S \subset \textsf{Ch}(t_n)} \frac{\sum_{v\in S}\bW(v)}{\bW} p_c^{|S|}(1-p_c)^{|\textsf{Ch}(t_n)|-|S|}\Pt\rk{\Ccal_{p_c}[n]=t}\\
        &= \sum_{k=1}^{|\textsf{Ch}(t_n)|} \sum_{S \subset \textsf{Ch}_k(t_n)} \frac{\sum_{v\in S}\bW(v)}{\bW}p_c^{k}(1-p_c)^{|\textsf{Ch}(t_n)|-k} \Pt\rk{\Ccal_{p_c}[n]=t} \\
        &= \sum_{k=1}^{|\textsf{Ch}(t_n)|}\sum_{v\in \textsf{Ch}(t_n)} \binom{|\textsf{Ch}(t_n)|-1}{k-1} \frac{\bW(v)}{\bW} p_c^{k}(1-p_c)^{|\textsf{Ch}(t_n)|-k} \Pt\rk{\Ccal_{p_c}[n]=t}.
    \end{split}
\end{align}
Now note that for any $v \in \bT$, we have that $W(v) = p_c \sum W(w)$, where the sum is over all children of $v$. Hence the last line above is equal to 
\begin{align}
    \begin{split}
        &\sum_{k=1}^{|\textsf{Ch}(t_n)|} \binom{|\textsf{Ch}(t_n)|-1}{k-1} \frac{\sum_{v\in t_n}\bW(v)}{\bW} p_c^{k-1}(1-p_c)^{|\textsf{Ch}(t_n)|-k} \Pt\rk{\Ccal_{p_c}[n]=t} \\
        &= \sum_{k=1}^{|\textsf{Ch}(t_n)|} \binom{|\textsf{Ch}(t_n)|-1}{k-1} p_c^{k-1}(1-p_c)^{|\textsf{Ch}(t_n)|-k} \mu_\bT\Big|_n[t]  = \mu_\bT\Big|_n[t],
    \end{split}
\end{align}
as required. The measures therefore extend by Kolmogorov's consistency theorem.
\end{proof}

Now let $\bZ_n$ be the number of vertices in the IIC at level $n$, rescaled by (i.e. divided by) a factor of $n^\beta$. Recall that we let $Y$ denote the random variable with Laplace transform appearing in \eqref{eqn:slack Laplace convergence}, and let $Y^*$ denote its size-biased version, meaning that (also using \eqref{eqn:Y mean}):
\begin{equation}
    \P \rk{Y^* \in [a,b]} = \frac{\E \ek{Y \mathbbm{1}\{Y \in [a,b]\}}}{\E \ek{Y}} = \CC\E \ek{Y \mathbbm{1}\{Y \in [a,b]\}}.
\end{equation}
The main aim of this section is to prove \cref{thm:IIC size biased}, which says that for $\bP$-almost every tree, $\bZ_n$ converges in distribution to $Y^*$.

Before launching into the proof, we give a short lemma, for which we use the following notation. For $n \geq 0, M \geq 0, K >0$ and $v \in \bT_n$, we write 
\[
p_v^{n,M} = \Pt \rk{v \lrd \bT_{n + M}}, \ \ \ p_v^{n,M}(K) = p_v^{n,M} \mathbbm{1}\{p_v^{n,M} \leq KM^{-\beta}\}\, .
\]

\begin{lemma}\label{lem:variance bound 2}
    There exists $K< \infty$ such that for all $n \geq 1$,
    \begin{equation}
        \Var{\Pt\rk{Y_n>0} \mathbbm{1}\{\Pt\rk{Y_n>0} \leq Kn^{-\beta}\}}\le \CC Kn^{-2\beta}\, ,
    \end{equation}
\end{lemma}
\begin{proof}
\begin{align}
\begin{split}
        \Var{\Pt\rk{Y_n>0} \mathbbm{1}\{\Pt\rk{Y_n>0} \leq Kn^{-\beta}\}} &\leq \bE \ek{\Pt\rk{Y_n>0}^2 \mathbbm{1}\{\Pt\rk{Y_n>0} \leq Kn^{-\beta}\}}  \\
        &\leq  Kn^{-\beta} \bE \ek{\Pt\rk{Y_n>0}} =\CC Kn^{-2\beta}.\qedhere
\end{split}
\end{align}
\end{proof}

This enables us to prove the following lemma, which is the key technical estimate in the proof of \cref{thm:IIC size biased}. In the following proof and for the rest of this section, we recall $\nk = k^{\frac{\sqrt{\alpha}+1}{\sqrt{\alpha}-1}}$ for all $k \geq 1$. 

\begin{lemma}[Subsequential control of connection probabilities]\label{lem:IIC subsequence convergence}
Fix $n \geq 0$ and an interval $(a,b) \subset (0, \infty)$ and suppose that $A_n$ is a sequence of (potentially random) sets such that $A_n \subset \bT_n$, such that $|A_n| \in [an^{\beta}, bn^{\beta}]$, and such that the law of $A_n$, denoted $\P_{\bT}^{A_n}$, is measurable with respect to $\Fcal_n$. Fix any $\delta>0$ and let $E_n^{\delta}$ denote the event that there is a subsequence $(M_j)_{j =1}^{\infty}$ such that \[
\P_{\bT}^{A_n}\left(\left|\frac{M_j^{\beta}}{n^{\beta}} \sum_{v \in A_n} p_v^{n,M_j} - |A_n|n^{-\beta}\CC \right| < \delta \text{ eventually as } j \to \infty \right) > 1-\delta.
\]
Then $E_n^{\delta}$ occurs eventually $\bP$-almost surely along the subsequence $(n_k)_{k=1}^{\infty}$.
\end{lemma}
\begin{proof}
First set $\delta' = \frac{\delta^2}{4b\CC}$ and choose $\tilde{K}< \infty$ such that $\bP \rk{\Pt\rk{Y_n>0} \geq \tilde{K} n^{-\beta}} \leq \delta'$ for all $n \geq 1$ (note that this is possible by \cref{thm:ConvergenceOfPartitionFunction}). We will abbreviate $A_n$ by $A$, let $\P_{\bT}^A$ denote the law of $A$, and write estimates conditionally on both $A$ and $\P_{\bT}^A$ (which is $\Fcal_n$-measurable by assumption) in what follows. 

First note that we can write for any $K \geq \tilde{K}$ and all sufficiently large $M$ that for any fixed $A' \subset \bT_n$ satisfying $|A'| \leq b n^{\beta}$,
    \begin{equation}
        \bE \ek{\left|\frac{M^{\beta}}{n^{\beta}} \sum_{v \in A'} p_v^{n,M}(K)  - |A'|n^{-\beta}\CC \right| \middle| \Fcal_n } \leq 2\delta' |A'|n^{-\beta}\CC \leq 2\delta' b\CC,
    \end{equation}
    by \eqref{eqn:slack prob convergence} and by our definition of $\tilde{K}$. Conditionally on $\Fcal_n$, the same estimate holds when $A$ is random with law $\P_{\bT}^A$ by averaging its possible realisations according to $\P_{\bT}^A$. Hence, since $\P_{\bT}^A$ is $\Fcal_n$-measurable, we can apply the tower property to deduce that 
    \begin{equation}
        \bE \ek{ \E_{\bT}^A \ek{\left|\frac{M^{\beta}}{n^{\beta}} \sum_{v \in A} p_v^{n,M}(K)  - |A|n^{-\beta}\CC \right|}} \leq 2\delta' b\CC.
    \end{equation}    
  Similarly, by \cref{lem:variance bound 2}, we have for any fixed $A' \subset \bT_n$ that
    \begin{equation}
        \Var{\frac{M^{\beta}}{n^{\beta}} \sum_{v \in A'} p_v^{n,M}(K) \big| \Fcal_n}  \leq \frac{M^{2\beta}}{n^{2\beta}}|A'|\CC KM^{-2\beta} \leq \frac{\CC b K}{n^{\beta}}.
    \end{equation}  
Again averaging over realisations of $A$ according to  $\P_{\bT}^A$, and then applying the tower property, we deduce that
    \begin{equation}
        \Var{ \E_{\bT}^A \ek{\frac{M^{\beta}}{n^{\beta}} \sum_{v \in A} p_v^{n,M}(K)} } \leq  \frac{\CC b K}{n^{\beta}}.
    \end{equation}
   Hence it follows from Chebyshev's inequality that
    \begin{align}\label{eqn:Chebyshev sum probs}
        \bP \rk{\E_{\bT}^A \ek{\left|\frac{M^{\beta}}{n^{\beta}} \sum_{v \in A} p_v^{n,M}(K) - |A|n^{-\beta}\CC \right|} \geq 4\delta' b\CC} \leq \frac{K}{4(\delta')^2 b \CC n^{\beta}} = O (n^{-\beta \eps})\, ,
    \end{align}
    where we have set $K=n^{\beta(1-\eps)}$ for some $\eps>0$ small enough that $\alpha (1-\eps)^2>1+\eps$ (e.g. $\eps = \frac{\sqrt{\alpha}-1}{\sqrt{\alpha}+1}$ will do - note that we then have that $K>\tilde{K}$ for all large enough $n$). By \cref{thm:ConvergenceOfPartitionFunction}, we know that for each $v \in \bT_n$, $M^{\beta}p_v^{n,M}$ converges almost surely to $\CC \bW_v$ as $M \to \infty$ (keeping $n$ fixed). Since $\bW$ is bounded in $\mathrm{L}^{\alpha(1-\eps)}(\bP)$, we have for all sufficiently large $M$ that there exists $C<\infty$, not depending on $n$, and an event $F_{M,n}$ that occurs eventually $\bP$-almost surely as $M\to \infty$, such that, for all $v \in \bT_n$,
    \begin{align}
        \bP \rk{F_{M,n}\text{ and } M^{\beta}p_v^{n,M} > K \middle| \Fcal_n} \leq CK^{-\alpha(1-\eps)}=C n^{-\beta\alpha(1-\e)^2}\le C n^{\beta(1+\e)}.
    \end{align}
    In particular, for all such $M$, we can write for any fixed $A' \subset \bT_n$ with $|A'| \leq bn^{\beta}$:
    \begin{align}
        \bP \rk{F_{M,n} \text{ and }  \exists v \in A': M^{\beta}p_v^{n,M} \geq K  \middle| \Fcal_n} \leq C|A'|K^{-\alpha(1-\eps)} \leq C|A'|n^{-\beta(1+\eps)} \leq Cb n^{-\beta \eps}.
    \end{align}
Again averaging over $A$ and applying the tower property as before, we deduce that
    \begin{align}\label{eqn:union v large p}
        \bE \ek{\1\{F_{M,n}\} \P_{\bT}^A \rk{\exists v \in A: M^{\beta}p_v^{n,M} \geq K }} \leq Cb n^{-\beta \eps}.
    \end{align}
In particular, \eqref{eqn:Chebyshev sum probs}, \eqref{eqn:union v large p}, Markov's inequality, a union bound Fatou's lemma and Borel--Cantelli imply that, eventually $\bP$-almost surely along the subsequence $(n_k)_{k=1}^{\infty}$, the following events occur simultaneously infinitely often as $M \to \infty$:
\begin{itemize}
    \item $\E_{\bT}^{A_n} \ek{\left|\frac{M^{\beta}}{n^{\beta}} \sum_{v \in {A_n}} p_v^{n,M}(K) - |A|n^{-\beta}\CC \right|} < 4\delta' b\CC$.
    \item $\1\{F_{M,n}\}\P_{\bT}^{A_n} \rk{\exists v \in {A_n}: M^{\beta}p_v^{n,M} \geq K } < \delta$.
\end{itemize}
Recall that $\delta^2=4\delta' b\CC$ and fix an $n$ and an $M$ such that the above two points occur. On the event $F_{M,n}$, we have that
\begin{align}
\begin{split}
 &\P_{\bT}^A \rk{\left|\frac{M^{\beta}}{n^{\beta}} \sum_{v \in A} p_v^{n,M} - |A|n^{-\beta}\CC \right| \geq \delta } \\
 &\leq \P_{\bT}^A \rk{\left|\frac{M^{\beta}}{n^{\beta}} \sum_{v \in A} p_v^{n,M}(K) - |A|n^{-\beta}\CC \right| > \delta} + \P_{\bT}^A \rk{\exists v \in A: M^{\beta}p_v^{n,M} \geq K } \leq 2\delta.
\end{split}
\end{align}
In particular, we know that $\bP$-almost surely along the subsequence, we have for all sufficiently large $k$ that there exists a sequence $(M_j^{(k)})_{j \geq 1}$ where the above estimate holds for $n=\nk$ and $M=M_j^{(k)}$. It therefore follows from Fatou's lemma that for all sufficiently large $k$,
\begin{align}
    \P_{\bT}^A \rk{\left|\frac{M^{\beta}}{\nk^{\beta}} \sum_{v \in A} p_v^{\nk,M} - |A|\nk^{-\beta}\CC \right| < \delta \text{ i.o. as } M \to \infty} \geq (1-2\delta)\liminf_{M \to \infty} \1\{F_{M,\nk}\}.
\end{align}
Since $F_{M,n_k}$ holds eventually $\bP$-almost surely as $M \to \infty$ for every $k$, this implies the claim.
\end{proof}

We now prove the main theorem in two steps as follows.

\begin{lemma}[Subsequential convergence of generation sizes in the IIC]\label{lem:IIC subsequence}
$\bP$-almost surely, the convergence of \cref{thm:IIC size biased} holds along the subsequence $(n_k)_{k \geq 1}$.
\end{lemma}
\begin{proof}
Note that our aim is equivalent to calculate for every $a<b$, the a.s. limit
    \begin{equation}
        \lim_{n\to\infty}\mu_{\bT}\rk{\bZ_n\in (a,b)}\, 
    \end{equation}
along the subsequence $(n_k)_{k \geq 1}$. By \cref{lem:IIC well defined}, we have for each $k \geq 1$ that
\begin{equation}
    \mu_{\bT}\rk{\bZ_{n_k}\in (a,b)}=\lim_{M\to\infty}\Pt\rk{Y\hk{{n_k}}\in (a,b)|0\lrd\bT_{{n_k}+M}}\, .
\end{equation}
Since we already know from \cref{lem:IIC well defined} that the limit as $M \to \infty$ exists $\bP$-almost surely, it will suffice to evaluate it along a subsequence.

We first write using Bayes' formula that $\Pt\rk{Y\hk{{n_k}}\in (a,b)|0\lrd\bT_{{n_k}+M}}$ is equal to
\begin{equation}\label{eqn:Bayes decomp}
 \frac{\Pt\rk{0\lrd\bT_{{n_k}+M}|Y\hk{{n_k}}\in (a,b)}\Pt\rk{Y\hk{{n_k}}\in(a,b) | 0\lrd\bT_{{n_k}}}\Pt\rk{0\lrd\bT_{{n_k}}}}{\Pt\rk{0\lrd \bT_{{n_k}+M}}}\, .
\end{equation}
We start with the first term and write $n$ in place of $n_k$ for the time being. Note that, by \cref{thm:ConvergenceOfPartitionFunction}, the expression $\frac{\Pt\rk{0\lrd\bT_{{n}+M}|Y\hk{{n}}\in (a,b)}\Pt\rk{0\lrd\bT_{{n}}}}{\Pt\rk{0\lrd \bT_{{n}+M}}}$ is equal to (the $O(\cdot)$ refers to the limit as $M \to \infty$ with $n$ fixed):
\begin{align}
\begin{split}
     & \frac{(1+o(1))(n+M)^{\beta}}{n^{\beta}}\sum_{\substack{A \subset \bT_{n}: \\ n^{-\beta}|A| \in (a,b)}} \Pt \rk{\OpenEd\hk{{n}} =A | Y\hk{{n}}\in (a,b)} \left( \sum_{v \in A} p_v^{n,M} + O(M^{-2\beta}) \right).
\end{split}
\end{align}
(The latter statement holds for all sufficiently large $n$ appearing in the subsequence, $\bP$-almost surely.) Recalling that in fact $n=n_k$ and applying \cref{lem:IIC subsequence convergence}, we deduce that, $\bP$-almost surely, the following holds for all sufficiently large $n$: with $\P_{\bT}$-probability at least $1-\delta$, there is a subsequence $(M^{(n)}_j)_{j \geq 1}$ along which this is equal to
\begin{align}
\begin{split}
     & (1+o(1))\sum_{\substack{A \subset \bT_{n}: \\ n^{-\beta}|A| \in (a,b)}} \Pt \rk{\OpenEd\hk{{n}} =A | Y\hk{{n}}\in (a,b)} \left( |A|n^{-\beta}\CC + O( \delta) + O((M^{(n)}_j)^{-\beta}) \right) \\
     &\to (1+O(\delta))\Et \ek{Y\hk{n} |  Y\hk{{n}}\in (a,b)}\CC.
\end{split}
\end{align}
All that remains is therefore to multiply this by $\Pt\rk{Y\hk{{n}}\in(a,b) | 0\lrd\bT_{{n}}}$, and then take the limit as $n \to \infty$. In particular, we see that $\bP$-almost surely, the limit of \eqref{eqn:Bayes decomp} (for all sufficiently large $n$ in the subsequence) along the subsequence $(M^{(n)}_j)_{j \geq 1}$ is equal to
\begin{equation}
    (1+O(\delta))\CC\Et \ek{Y\hk{n} |  Y\hk{{n}}\in (a,b)} \Pt\rk{Y\hk{{n}}\in (a,b) | 0\lrd\bT_{{n}}}.
\end{equation}
In particular, since we already know that the limit exists (by \cref{lem:IIC well defined}), we see that the limit of \eqref{eqn:Bayes decomp} (for fixed $n$) is
\begin{equation}
(1+O(\delta))\CC\Et \ek{Y\hk{n} \mathbbm{1}\{ Y\hk{{n}}\in (a,b)\} | Y_n > 0}.
\end{equation}
So to establish the subsequential convergence, it just remains to take the limit of this latter expression. However, we already know from \cref{thm:ConvergenceOfLaplaceTransform} that $Y\hk{n}$ converges to $Y$, and hence this converges to the size-biased version $Y^*$ appearing in \cref{thm:IIC size biased}. Since $\delta>0$ was arbitrary, this implies the result.
\end{proof}

It remains to lift this from the subsequence $(n_k)_{k \geq 1}$ to the whole sequence $n \geq 1$.

\begin{proof}[Proof of \cref{thm:IIC size biased}]
It is sufficient to evaluate the limit (as $n \to \infty)$ of the limit (as $M \to \infty$) of $\frac{M^\beta}{n^{\beta}}\Pt\rk{0\lrd\bT_{{n}+M}|Y\hk{{n}}\in (a,b)}$, since we already know that the other terms appearing in \eqref{eqn:Bayes decomp} converge, $\bP$-almost surely. To this end, we take $E_n$ as in \eqref{EquationEn}, i.e.
\begin{align}
 E_n=\gk{\abs{Y_n-\Ynk}>Y_n^{1-\delta}\vee n^{\beta-\delta}},
\end{align} 
but this time take $\eps>0$ and 
\begin{equation}
    F_n'=F'_n(\bT,\e)=\gk{\Pt\rk{E_n | \Gcal_n^\bT, Y\hk{{n}}\in (a,b)}<\e/4}\, .
\end{equation}
Since the event $\{Y\hk{{n}}\in (a,b)\}$ is $\Gcal_n^\bT$-measurable, it follows from \cref{prop:verify condition} that 
\begin{equation}\label{eqn:Fn' LB}
    \bP\rk{F_n'|\Fcal_n}\ge 1-o(1)
\end{equation}
as well. Our strategy is therefore to follow that of \cref{lem:lift to whole sequence}, and first show an analogous statement to \eqref{eqn:A F B inclusion} and \cref{LemmaLink}, that is that on the event $F_n'$ the $n^{th}$ quantity should be close to the $n_k^{th}$ quantity when $n \in (n_{k-1}, n_k]$, and then apply the same logic as \cref{lem:lift to whole sequence} to lift the result of \cref{lem:IIC subsequence} to the whole sequence.

To this end, first note that on the event $F_n'$ we have by \cref{thm:ConvergenceOfLaplaceTransform} that for all $\eps>0$ that there exists $\delta>0$ such that
\begin{align}\label{eqn:conditioning n nk}
\begin{split}
    \Pt\rk{ Y\hk{{n_k}}\in (a,b)| Y\hk{{n}}\in (a,b)} &\geq     \Pt\rk{ Y\hk{{n_k}}\in (a,b)| Y\hk{{n}}\in (a+\delta,b-\delta)} (1-\eps/4) \\
    &\geq 1-\eps/2.
\end{split}
\end{align}
Hence we deduce that on the event $F_n'$, for any $n \in (n_{k-1}, n_k]$,
\begin{align}\label{eqn:ext IIC prob LB}
\begin{split}
    \Pt\rk{0\lrd\bT_{{n}+M}|Y\hk{{n}}\in (a,b)} &\geq \Pt\rk{0\lrd\bT_{{n_{k}}+M}|Y\hk{{n}}\in (a,b)} \\
    &\geq \Pt\rk{0\lrd\bT_{{n_{k}}+M}|Y\hk{{n_k}}\in (a,b)} (1-\eps/2).
\end{split}
\end{align}

In the other direction, first note that similarly to \eqref{eqn:conditioning n nk}, on the event $F_n'$ we have that \[
\Pt\rk{ Y\hk{{n_k}}\notin (a,b)| Y\hk{{n}}\notin (a,b)} \geq 1-\eps/2,
\]
which rearranges to 
\begin{align}\label{eqn:conditioning n nk 2}
\begin{split}
    \Pt\rk{ Y\hk{{n}}\in (a,b)| Y\hk{{n_k}}\in (a,b)} &\geq 1-\eps/2.
\end{split}
\end{align}
Hence we deduce that on the event $F_n'$, for any $n \in (n_{k-1}, n_k]$,
\begin{align}
\begin{split}
\Pt\rk{0\lrd\bT_{{n}+M}|Y\hk{{n}}\in (a,b)} &\leq \Pt\rk{0\lrd\bT_{{n}+M}|Y\hk{{n_k}}\in (a,b)} (1 + \eps/2) \\
    &= \Pt\rk{0\lrd\bT_{{n_{k}}+(M+n-n_k)}|Y\hk{{n_k}}\in (a,b)} (1+ \eps/2)\, .
\end{split}
\end{align}
Taking limits as $M\to \infty$ gives 
\begin{align}\label{eqn:ext IIC prob LB 2}
\begin{split}
\lim_{M \to \infty} M^{\beta} \Pt\rk{0\lrd\bT_{{n}+M}|Y\hk{{n}}\in (a,b)} &\leq \lim_{M \to \infty} M^{\beta} \Pt\rk{0\lrd\bT_{{n_{k}}+M}|Y\hk{{n_k}}\in (a,b)} (1+ \eps/2)\, .
\end{split}
\end{align}
To summarise, and using the fact that $\frac{n}{n_k} \to 1$ as $n \to \infty$, we showed in \eqref{eqn:ext IIC prob LB} and \eqref{eqn:ext IIC prob LB 2} that on the event $F_n'$ (and provided $n$ is sufficiently large), 
\begin{align}
\begin{split}
\left|   \lim_{M \to \infty} \frac{M^{\beta}}{n^{\beta}}\Pt\rk{0\lrd\bT_{{n}+M}|Y\hk{{n}}\in (a,b)} - \lim_{M \to \infty} \frac{M^{\beta}}{n_k^{\beta}} \Pt\rk{0\lrd\bT_{{n_{k}}+M}|Y\hk{{n_k}}\in (a,b)} \right|
    &\leq \eps.
\end{split}
\end{align}
We can therefore use this latter statement in combination with \eqref{eqn:Fn' LB} to imitate the proof of \cref{lem:lift to whole sequence} and deduce that 
\begin{equation}
  \lim_{M \to \infty} \left| \frac{M^{\beta}}{n^{\beta}}  \Pt\rk{0\lrd\bT_{{n}+M}|Y\hk{{n}}\in (a,b)} - \CC\Et \ek{Y\hk{n} \mathbbm{1}\{ Y\hk{{n}}\in (a,b)\} | Y_n > 0}\right| \leq \eps\, ,
\end{equation}
for all sufficiently large $n$, $\bP$-almost surely, 
and then combine with \eqref{eqn:Bayes decomp} exactly as in the previous proof to deduce the result.
\end{proof}

\appendix

\section{}

\begin{lemma}\label{LemmaConstants}
    Under \cref{assumption} (i.e. $\bP\rk{\abs{\bT_1}>x}\sim c_1 x^{-\alpha}$ for $\alpha\in (1,2)$ or $\Var{\abs{\bT_1}}<\infty$), we have that $\P\rk{Y_1>x}\sim c_1 \mu^{-\alpha}x^{-\alpha}$, as well as 
    \begin{equation}
    \phi(\theta)=1-\CC^{-1}\theta (1+(\CC^{-1}\theta)^{\alpha-1})^{-\beta}\quad\text{with}\quad \CC = c^{-\beta}_1\mu^{\alpha\beta}\Gamma(1-\alpha)^{-\beta}\beta^\beta\, .
    \end{equation}
Moreover, the branching mechanism (see \cref{sec:intro}) satisfies
    \begin{equation}
        \psi(\l)=\begin{cases}
            \beta C_\alpha^{1-\alpha} \lambda^{\alpha} &\text{ if }\alpha<2\, ,\\
           C_\alpha^{-1}\l^2&\text{ if the variance is finite.}
        \end{cases}
    \end{equation}
\end{lemma}
\begin{proof}
    According to \cite[Lemma 2]{slack1968branching}, if the generating function $f(s)=\E\ek{s^{Y_1}}$ of the critical annealed process satisfies
    \begin{equation}\label{Equation231120231}
        f(s)=s+(1-s)^{\alpha}L(1-s)\quad\text{with}\quad \lim_{x\to 0^+}L(x)=c_o\, ,
    \end{equation}
    then
    \begin{equation}
        \P\rk{Y_n>0}^{\alpha-1} c_o\sim \rk{n(\alpha-1)}^{-1}\quad\Longleftrightarrow\quad \P\rk{Y_n>0}\sim \rk{c_o n(\alpha-1)}^{-\beta}\, .
    \end{equation}
    Write $C_\alpha=C(\alpha, c_o)= c_o^{-\beta}\beta^\beta$ so that $\P\rk{Y_n>0}\sim \CC n^{-\beta}$. It is then shown that (see \cite[Theorem 1]{slack1968branching})
    \begin{equation}
        \E\ek{\ex^{-u n^{-\beta}Y_n}|Y_n>0}\sim 1-(u/C_\alpha)\rk{1+(u/C_\alpha)^{\alpha-1}}^{-\beta}\, .
    \end{equation}
    By Equation \eqref{Equation231120231}, we get that, as $u \to 0$,
    \begin{equation}
        \E\ek{\ex^{-u Y_1}}=f\rk{\ex^{-u}}=\ex^{-u}+\rk{1-\ex^{-u}}^\alpha L\rk{1-\ex^{-u}} = 1 - u + c_ou^{\alpha} + o(u^{\alpha})\, .
    \end{equation}
    This gives by \cite[Theorem 8.1.6]{bingham1989regular}, that
    \begin{equation}
        \P\rk{Y_1>x}\sim \frac{c_o}{\Gamma(1-\alpha)}x^{-\alpha}\, .
    \end{equation}
    Now, suppose that we have in the quenched case that the mean $\mu>1$ and that
    \begin{equation}
        \bP\rk{\abs{\bT_1}>x}\sim c_1 x^{-\alpha}\, .
    \end{equation}
By standard concentration estimates for binomial random variables, it follows that $\P\rk{Y_1>x}\sim c_1\mu^{-\alpha}x^{-\alpha}$. This implies that, in the notation above,
    \begin{equation}
        c_o=c_1\mu^{-\alpha}\Gamma(1-\alpha)\quad\text{and hence}\quad C_\alpha=c^{-\beta}_1\mu^{\alpha\beta}\Gamma(1-\alpha)^{-\beta}\beta^\beta\, .
    \end{equation}
    Next, we compute the branching mechanism $\psi$. Note that by \eqref{eqn:annealed Poisson mgf}, if
\begin{align}\label{eqn:Laplace computations}
\begin{split}
    \E \ek{e^{-\theta \tY_t} \vert \tY_s = a}=\ex^{-a u_{t-s}(\theta)}\qquad\text{then}\qquad u_t(\l)&=\l\rk{1+\rk{\CC^{-1} \l}^{\alpha-1}t}^{-\beta} \\
    \text{ and} \qquad  u_t'(\l)&=-\beta \CC^{1-\alpha}\l^{\alpha} \rk{1+\rk{\CC^{-1} \l}^{\alpha-1}t}^{-\beta-1} \, .
\end{split}
\end{align}
By \cite[Section 1.3]{le1999spatial}, we have that $\psi(\l)=\tilde{c}\l^{\alpha}$ (in the case $\alpha<2$) and $\psi(\l)=\tilde{c}\l^2$ in the finite variance case. Furthermore, we have that $u_t(\l)$ satisfies the following integral equation:
\begin{equation}
    u_t(\l)+\int_0^t \psi\rk{u_s(\l)}\d s=\l\, .
\end{equation}
Then, if $\alpha<2$
\begin{equation}
    u_t(\l)+\tilde{c}\int_0^t \rk{u_s(\l)}^\alpha\d s=\l\quad\Longrightarrow\quad u_0'(\l)+\tilde{c}u_0(\l)^\alpha=0\, ,
\end{equation}
From this, \eqref{eqn:Laplace computations} and the fact that $\frac{\beta+1}{\beta}=\alpha$ we infer that $\tilde{c}=\beta C_\alpha^{1-\alpha}$. The finite variance case is analogous.
\end{proof}
\begin{lemma}\label{LemmaAnnealedLargeDeviation}
    Let $\P$ be the law of the critical annealed tree with stable law $\alpha\in (1,2]$. Let $Y_l$ be the number of individuals at generation $l$ conditioned to survive and let $(Y_l\hk{i})_{i\geq 1}$ be i.i.d. copies of $Y_l$. Then for all $\e>0$, there exists $C>0$ such that for all $N,l,x$ with $x\ge N^{1/(\alpha-\e)}$
    \begin{equation}
        \P^{\otimes N}\rk{\sum_{i=1}^N (p_lY_l\hk{i}-1)>x}< C  N x^{-(\alpha-\e)}\, ,
    \end{equation}
    for $p_l:=\P\rk{Y_l>0}\sim \CC l^{-\beta}$. 
\end{lemma}
\begin{proof}
    We begin by developing a bound for 
    \begin{equation}
        \P\rk{ l^{-\beta}Y_l>z}\, .
    \end{equation}
    Let $f_n$ be the probability generating function of the critical process after $n$ generations, i.e. $f_n(s)=\E\ek{s^{Y_n}}$. Note that, by \eqref{eqn:slack prob convergence} (see also \cite[Lemma 2]{slack1968branching}) there exists $e_n=o(1)$ such that 
    \begin{equation}
        1-f_n(0)=\CC n^{-\beta}\rk{1+e_n}\, .
    \end{equation}
    Abbreviate $c=\CC^{\beta }$ from now on. We now assume that $0<u<1$. Define
    \begin{equation}
        y_n(u)=y_n=\ex^{-u(1-f_n(0))}\, .
    \end{equation}
    Hence, there exist sequences $\rk{h_n}_n=o(1)$ (independent of $u$, as $u$ remains bounded from above) and $\rk{k_n}_n\in\N$ such that
    \begin{equation}
    k_n=k_n(u)=u^{-1/\beta}n\rk{1+h_n}\qquad\text{and}\qquad f_{k_n}(0)\le y_n<f_{k_n+1}(0)\, ,
    \end{equation}
    By \cite[Eq. (2.6)]{slack1968branching} that there exists $\rk{g_n}_n=o(1)$ such that
    \begin{equation}
        \frac{1-f_{r+1}(0)}{1-f_r(0)}=1-\beta r^{-1}\rk{1+g_r}\, ,
    \end{equation}
    for all $r\geq 1$, where we have used the relation $\CC^{\alpha-1}c_o=\beta$ from the previous lemma. Hence
    \begin{equation}
        \frac{1-f_{n+M}(0)}{1-f_n(0)}=\exp\rk{\sum_{l=0}^{M-1}\log\rk{1-\beta (n+l)^{-1}\rk{1+g_{n+l}}}}.
    \end{equation}
    Therefore, there exists a $\rk{j_n}_n=o(1)$ such that for all $n, M \geq 1$:
    \begin{equation}\label{eqn:f ratio}
        \frac{1-f_{n+M}(0)}{1-f_n(0)}=\ex^{-\beta \rk{\log(M+n-1)-\log(n))}\rk{1+j_n}}\, .
    \end{equation}
    Now let $\phi_n(t)$ be the Laplace transform of (the rescaled) size of generation $n$ conditioned on survival, so that by definition,
    \begin{equation}
        \phi_n(u)=\E\ek{\ex^{-u(1-f_n(0))Y_n}|Y_n>0}=1-\frac{1-f_n(y_n(u))}{1-f_n(0)}\, .
    \end{equation}
   By definition of $k_n$, we moreover have that
    \begin{equation}
         \frac{1-f_{n+k_n(u)}(0)}{1-f_n(0)}\ge \frac{1-f_n(y_n)}{1-f_n(0)}\ge \frac{1-f_{n+k_n(u)+1}(0)}{1-f_n(0)}\, ,
    \end{equation}
    Hence, by \eqref{eqn:f ratio}, we have for all $0<u<1$
    \begin{align}\label{eqn:Ysurvive rescaled tails}
    \begin{split}
        \phi_n(u) = 1-\left(\frac{n}{n+k_n(u)(1+o(1))}\right)^{\beta (1+o(1))}&= 1- \left(\frac{1}{1+u^{-1/\beta}(1+o(1))}\right)^{\beta (1+o(1))} \\
        &= 1-u\rk{1+u^{\alpha-1}}^{-\beta(1+o(1))}\, ,
    \end{split}
    \end{align}
    where the $o(1)$ are independent of $u$.
        Note that if $X$ is a positive random variable with mean 1, then
\begin{equation}
    \gk{X\ge 2r}\subset \gk{\ex^{-X/r}-1+X/r\ge 1}\, ,
\end{equation}
and moreover the latter random variable is non-negative. Hence, using Markov's inequality, we obtain that
\begin{equation}\label{Eq1710231}
    \P(X\ge 2r)\le \E\ek{\ex^{-X/r}}-1+\E\ek{X/r}=\phi_X(1/r)-1+1/r\, .
\end{equation}
Hence, plugging Equation \eqref{Eq1710231} into
\eqref{eqn:Ysurvive rescaled tails}, we get that there exists a $C>0$ such that for all $z>1$, as $l\to\infty$
    \begin{equation}\label{Equation72823}
          \P\rk{ p_lY_l>z}\le C z^{-\alpha(1+o(1))}\, .
    \end{equation}
    Using \cite[Theorem 5.1 (ii)]{berger2019notes} (with $y=x$ in Berger's notation) Equation \eqref{Equation72823} implies that for every $\tilde\alpha<\alpha$ that there exists an $l_o>0$ and $C>0$ such that for all $l>l_o$ and $n\ge 1$ and for all $x\ge n^{1/\tilde\alpha}$, we have that
    \begin{equation}\label{Equation728231}
        \P\rk{ \sum_{i=1}^n\rk{p_lY_l\hk{i}-1}>x}\le C n x^{-\tilde\alpha}\, .
    \end{equation}
    However, note that (by adjusting the constant $C$) Equation \eqref{Equation728231} continues to hold for $l\in\gk{1,\ldots,l_o}$. Indeed, as $l_o$ is fixed, the probability generating function of $Y_k$ is given $f^{\circ k}(x)=f(f(\ldots f(x)\ldots))$, where we apply $f$ exactly $k$ times. Since this is a finite convolution, we get the same tail bounds as in the case $l=1$. This concludes the proof.
\end{proof}

\bibliography{thoughts}
\bibliographystyle{alpha}
\end{document}